\documentclass[a4paper, 12pt]{amsart}
\usepackage{amsmath, amsthm, amscd, amssymb, amsfonts, amssymb, latexsym}
\usepackage{enumerate}
\usepackage{verbatim}
\usepackage{float}

\usepackage{lmodern}

\usepackage{tikz}

\usepackage{array}
\usepackage{hyperref}
\hypersetup{colorlinks = true,	allcolors  = blue}

\newcommand{\G}{\Gamma}
\newcommand{\Z}{\mathbb{Z}}
\newcommand{\R}{\mathbb{R}}
\newcommand{\C}{\mathbb{C}}
\newcommand{\N}{\mathbb{N}}
\newcommand{\ff}{\mathbb{F}}
\newcommand{\Tr}{\mathrm{Tr}}

\newcommand{\sk}{\smallskip}
\newcommand{\msk}{\medskip}

\newtheorem{thm}{Theorem}[section]
\newtheorem{prop}[thm]{Proposition}
\newtheorem{lem}[thm]{Lemma}
\newtheorem{coro}[thm]{Corollary}

\theoremstyle{definition}
\newtheorem{rem}[thm]{Remark}
\newtheorem{exam}[thm]{Example}
\newtheorem{defi}[thm]{Definition}

\theoremstyle{remark}
\newtheorem*{note}{Note}
\newtheorem*{notation}{Notation}

\usepackage{color}

\begin{document} \sloppy

	\numberwithin{equation}{section}
	\title[The nature of the spectrum of GP-graphs and weak Waring numbers]{The nature of the spectrum of generalized Paley graphs and weak Waring numbers over finite fields}
	\author[R.A.\@ Podest\'a, D.E.\@ Videla]{Ricardo A.\@ Podest\'a, Denis E.\@ Videla}
	\dedicatory{\today}
	\thanks{Partially supported by CONICET and SECyT-UNC}
	
	\address{Ricardo A.\@ Podest\'a, FaMAF -- CIEM (CONICET), Universidad Nacional de C\'ordoba, \newline
		Av.\@ Medina Allende 2144, Ciudad Universitaria, (5000) C\'ordoba, Argentina. 
		\newline {\it E-mail: podesta@famaf.unc.edu.ar}}
	\address{Denis E.\@ Videla, FaMAF -- CIEM (CONICET), Universidad Nacional de C\'ordoba, \newline
		Av.\@ Medina Allende 2144, Ciudad Universitaria,  (5000) C\'ordoba, Argentina. 
		\newline {\it E-mail: devidela@famaf.unc.edu.ar}}

	\begin{abstract}  
		We consider the family of generalized Paley graphs (GP-graphs for short) $\G(k,q)=Cay(\ff_q, (\ff_q^*)^k)$, with $q=p^m$ and $p$ prime. We characterize all GP-graphs having real spectrum; namely, $Spec(\G(k,q)) \subset \R$ if and only if $\G(k,q)$ is undirected. 
		We then study conditions for integrality in the spectrum and give a general method to produce integral GP-graphs through cyclotomic polynomials. Using this, we construct several infinite families of integral GP-graphs.
		Next, we focus on directed GP-graphs (GP-digraphs). 
		We show that GP-digraphs always have three or more eigenvalues, and then we prove that there is only one kind of GP-digraphs having three different eigenvalues: the oriented Paley graphs $\vec{\mathcal{P}}_q$ or disjoint unions of copies of them, $\vec{\mathcal{P}}_q \cup \cdots \cup \vec{\mathcal{P}}_q$. 
		Then, we show that generically the GP-digraphs have period 1 (equivalently index of imprimitivity 1) except for $\G(q-1,q)$ with $q$ odd, which is the disjoint union of oriented $p$-cycles, having period $p$. 
		Finally, as an application, we study weak Waring numbers over finite fields through GP-graphs. In particular, we reduce the computation of the weak Waring numbers over finite fields to the computation of classic Waring numbers over finite fields, a result previously obtained by Cochrane and Cipra in 2012 by other means.
	\end{abstract}

	\maketitle

	\section{Introduction} 
	In this work, we study the nature of the spectrum of the family of generalized Paley (GP) graphs, completing the characterization of those GP-graphs which have integral (already known), real, or complex spectrum. The situation for these graphs is, unlike for general families of graphs, very neat and simple: the GP-graphs with real spectrum are exactly those undirected GP-graphs. As an application, we solve the question of the weak Waring numbers over finite fields.

	From now on, $\ff_q$ will denote a finite field with $q$ elements, where $q=p^m$ with $p$ prime and $m \in \N$.
	A \textit{generalized Paley graph} is the Cayley graph
	\begin{equation} \label{eq: Gkq}
		\G(k,q) = Cay(\ff_{q},R_k) \qquad \text{with} \qquad R_k=\{ x^{k} : x \in \ff_{q}^*\} = (\ff_q^*)^k.
	\end{equation}
	Since $\G(k,q)=\G(\gcd(k,q-1),q)$, 
	one assumes that $k\mid q-1$ (and we do this from now on).
	The graph $\G(k,q)$ is $n$-regular with 
	$$n=\tfrac{q-1}{k}.$$ 
	It is well-known that $\G(k,q)$ is connected if and only if $ord_n(p) = m$, and  
	that $\G(k,q)$ is undirected if and only if $q$ is even or $q$ is odd and $k \mid \frac{q-1}2$.

	The spectrum of an arbitrary graph $\G$, denoted $Spec(\G)$, is the set of eigenvalues of its adjacency matrix $A$ counted with multiplicities.
	If $\Gamma$ has different eigenvalues $\lambda_0, \ldots, \lambda_t$ with multiplicities $m_0,\ldots,m_t$, we write 
	as usual 
	\begin{equation} \label{eq: spec}
		Spec(\Gamma) = \{[\lambda_0]^{m_0}, \ldots, [\lambda_t]^{m_t}\}.
	\end{equation}
	It is well-known that if $\G$ is an $n$-regular graph, then $n$ is the greatest eigenvalue, with multiplicity equal to the number of connected components of $\G$. 
	Thus, $\G$ is connected if and only if $n$ has multiplicity $1$.
	For $n$-regular digraphs, i.e., those directed graphs such that any vertex has the same in-degree and out-degree equal to $n$, $\G$ is strongly connected if and only if $n$ has multiplicity $1$.

	The spectrum of $\G$ is said to be \textit{real} or \textit{integral} if 
	$$ Spec(\Gamma)\subset\mathbb{R} \qquad \text{or} \qquad Spec(\Gamma)\subset\mathbb{Z},$$ 
	respectively. 
	The spectra of a few families of GP-graphs are known. The spectrum of the graphs $\G(k,q)$ with small $k$ are known; the cases $k=1,2$ are classic, and for $k=3,4$  (also $k=5$ in some cases) they were recently computed in \cite{PV3}. 
	The spectrum of semiprimitive GP-graphs is computed in \cite{PV3} (see also \cite{PV4b} for a subfamily and \cite{BWX} for its generalization) and the spectrum of Hamming GP-graphs can be found in \cite{PV4}.

	In some of our previous works (see \cite{PV4b}, \cite{PV3}, \cite{PV18}) we have studied certain structural and spectral properties of GP-graphs. 
	For instance we have studied particular families such as $\G(3,q)$ and $\G(4,q)$, or the strongly regular and semiprimitive GP-graphs and the subfamily $\G(q^\ell+1, q^m)$ with $\ell \mid m$ (see \cite{PV4b}). 
	We have computed the spectrum and energy, we gave constructions of equienergetic non-isospectral pairs and we characterized Ramanujan families. 
	Recently, in \cite{PV18}, we studied two basic structural properties: connectedness and bipartiteness. We gave the connected components of $\G(k,q)$ and we show that any $\G(k,q)$ is non-bipartite except for the graphs $\G(2^m-1,2^m)$ with $m\in \N$ which are $2^{m-1}$ copies of $K_2$.   
	Previously, in \cite{PV3}, we have characterized those GP-graphs having integral spectrum.
	In this work, we complete the study of the nature of the spectrum, characterizing those GP-graphs having real spectrum or with complex non-real spectrum. For directed (i.e.\@ oriented) GP-graphs we study the periods and the index of imprimitivity.

	Given $k\in \N$, the \textit{Waring number} $g(k,q)$ over the finite field $\ff_q$ is the minimum $s\in \mathbb{N}$ (if exists) such that for any element $a \in \mathbb{F}_q$ there exist $x_1, \ldots, x_s \in \mathbb{F}_q$ with
	$$ a= x_1^k + \cdots + x_s^k. $$ 
	We have studied these numbers before in \cite{PV6}, \cite{PV7} and \cite{PV9}.
	In a similar way, the \textit{weak Waring number} over finite fields, denoted $w(k,q)$, is the minimum $s\in \mathbb{N}$ (if exists) such that for any $a\in\mathbb{F}_{q}$ there exist 
	$x_1, \ldots, x_s \in \mathbb{F}_q$  
	such that
	$$ a= \pm x_1^k \pm \cdots \pm x_s^k,$$
	meaning that each term can have a plus or a minus sign independently. 
	Here, we will express $w(k,q)$ in terms of $g(k,q)$.

	\subsubsection*{Outline and results}
	Briefly, the paper is organized as follows. In the first two sections, apart from the Introduction, we study the nature of the spectrum of some Cayley and GP-graphs. In Section~\ref{sec: 4} we give a list of families of GP-graphs with known spectrum and check the conditions for the spectrum to be integral, real or complex. In the next section we study integral GP-graphs in more detail. In the next two sections, we study directed GP-graphs. In Section \ref{sec:3}, we characterize all GP-digraphs having exactly 3 eigenvalues, while in Section \ref{sec: GP-digraphs} we study the periods of GP-digraphs and their relation with the spectrum.
	In Section \ref{sec: Waring}, we give an interesting application to weak Waring numbers, a refinement of Waring numbers. Finally, in the last section, we exhibit some worked examples. 
	
	More precisely, in this work, we do the following. 
	In Section \ref{sec:2a}, we study the nature of the spectrum (real vs complex) for general Cayley graphs $X(G,S)$ with the assumption $S\cap S^{-1} = \varnothing$ (i.e.\@ non-mixed graphs), in terms of directedness. For a set $S$ satisfying this condition, in Theorem \ref{thm: XGSy-S} we show that 
	$$X(G,\check S) = X(G,S) \cup X(G,S^{-1}) = \overset{_{\rightarrow}}{X}(G,S) \cup \overset{_{\leftarrow}}{X}(G,S),$$ 
	where $\overset{_{\rightarrow}}{X}(G,S)$ denotes $X(G,S)$ with the given orientation and $\overset{_{\leftarrow}}{X}(G,S)$ denotes $X(G,S)$ with the reverse orientation, and that for $G$ abelian, the spectra of these graphs satisfy 
	$$Spec(X(G,\check S)) = 2 \, \mathrm{Re} (Spec(X(G,S))).$$ 
	
	In Section \ref{sec:2}, we use Theorem \ref{thm: XGSy-S} to study the nature of the spectrum of GP-graphs $\G(k,q)$. 
	In Theorem~\ref{thm: real spec} we prove that $Spec(\G(k,q)) \subset \R$ if and only if $\G(k,q)$ is undirected. 
	This completes the characterization of the nature of the spectrum of GP-graphs $ \G(k,q)$ in terms of the parameters. Namely, binary GP-graphs (i.e., defined over $\ff_{2^m}$) are always integral (and undirected); that is,
	$Spec(\G(k,2^m)) \subset \Z$
	for every $m$ and every $k\mid 2^m-1$. For any $q$ odd and any $k \mid q-1$ one has
	$$ Spec(\G(k,q)) \subset \R \: \Leftrightarrow \: k\mid \tfrac{q-1}2 \qquad \text{and} \qquad 
	Spec(\G(k,q)) \subset \Z \: \Leftrightarrow \: k \mid \tfrac{q-1}{p-1}.$$ 
	We also study the nature of the spectrum of GP-graphs arithmetically. In Theorem~\ref{thm: nature spec v2} we obtain a similar kind of result as in Theorem \ref{thm: XGSy-S} but for GP-graphs. In particular, under certain conditions, we compare the graph $\G(\frac k2,q)$ and its complex spectrum with the two oriented versions of $\G(k,q)$ and their real spectrum. 
	In Corollary \ref{coro: wkqs exp} we completely determine the nature of the spectrum of $\G(k,q)$ --integral, real, or complex-- in terms of the divisors of $k$. In particular, for each odd $q$ fixed, we obtain the exact number of GP-graphs $\G(k,q)$ having a complex or real spectrum. Namely, if $q=2^t r+1$ and $r=p_1^{e_1} \cdots p_s^{e_s}$ is the prime factorization of $r$, then there are 
	$$ N=(e_1+1) \cdots (e_s+1) $$
	directed GP-graphs $\G(2^{t}s, q)$ with $s\mid r$ (complex spectrum) and $tN$ undirected GP-graphs $\G(2^{t'}s, q)$ with $t'<t$ and $s\mid r$ (real spectrum).
	
	In the next section, we give a list of families of GP-graphs with known spectrum, and we recall the eigenvalues in each case. Then we check that the nature of the spectrum corresponds to the conditions obtained in the previous section. 
	In particular, we recall the GP-graphs with small $k$, such as $\G(k,q)$ with $k=1,2,3,4$, and GP-graphs with $k$ not fixed, such as cycles (directed and undirected), Hamming GP-graphs $\G(\frac{p^{bm}-1}{b(p^m-1)}, p^{bm})$, and semiprimitive GP-graphs (see Examples \ref{exam: complete}--\ref{exam: semip}).

	In Section \ref{sec: integral}, we study integral GP-graphs in more detail. In Corollary \ref{coro: integral GPs} we express the condition of integrality of $\G(k,q)$ in terms of the $p_i$-adic valuations of $k$, for every prime divisor $p_i$ of $\frac{q-1}{p-1}$. We also compute the number of integral GP-graphs over $\ff_q$, complementing similar results obtained 
	in Corollary \ref{coro: wkqs exp} for the number of GP-graphs with real or complex spectrum.  
	In Propositions \ref{prop: finite families of integral GPs} and \ref{prop: families of integral GPs} we give infinite families of integral GP-graphs. In Proposition \ref{prop: towers of integral GPs} we show that given an integral GP-graph $\G(k,q)$ then $\G(k\frac{q^a-1}{q-1}, q^a)$ is also integral for every $a\in \N$. In Proposition \ref{prop: cyclotomic} we use cyclotomic polynomials $\Phi_d(x)$ to get integral GP-graphs of the form $\G(\Phi_d(p), p^dt)$ with $d,t \in \N$. In Theorem \ref{thm: general integral GP-graphs} we apply Proposition \ref{prop: towers of integral GPs} to Propositions \ref{prop: finite families of integral GPs}, \ref{prop: families of integral GPs} and \ref{prop: cyclotomic} to get four different 2-parameter infinite families of integral GP-graphs of the form
	$$\{ \G(k \tfrac{q^at-1}{q^t-1}, q^{at})\}_{a,t \in \N},$$
	where $k$ and $q$ satisfy certain mild conditions.

	The next two sections are devoted to directed GP-graphs.
	In Section \ref{sec:3}, we study the spectrum of GP-digraphs. 
	In Theorem \ref{thm: directed SRG}, we show that any GP-digraph has at least 3 different eigenvalues and then we characterize all GP-digraphs having exactly 3 eigenvalues. They are specifically the directed Paley graph or disjoint unions of it. More precisely, a directed GP-graph $\G=\G(k,p^m)$ has 3 different eigenvalues if and only if $k=2 \frac{p^m-1}{p^a-1}$ where $a=ord_n(p)$ with $n=\frac{p^m-1}{k}$ and $p^a \equiv 3 \pmod 4$. That is, in such case we have
	$$ \G=\G(2\tfrac{p^m-1}{p^{a}-1}, p^m) = \vec{\mathcal{P}}_{p^a} \sqcup \cdots \sqcup \vec{\mathcal{P}}_{p^a} \qquad 
	(\text{$p^{m-a}$ times}).$$
	The spectra of these graphs are given explicitly in the theorem.

	In Section \ref{sec: GP-digraphs}, we study the periods of GP-digraphs. 
	For a digraph $G$, the \textit{period} of $G$ is the greatest common divisor between all the lengths of the directed cycles in $G$. In Theorem~\ref{thm: GPperiods} we show that any GP-digraph $\G(k,q)$ has period 1, except for $\G(q-1,q)$, with $q=p^m$, which are disjoint unions of directed cycles $\vec{C}_p$, hence having period $p$. 
	As a consequence, in Proposition \ref{prop: real spec S1} we show that for any directed GP-graph $\G(k,q)$ we have
	$$ Spec(\G(k,q)) \cap \tfrac{q-1}{k} \mathbb{S}^1  = \{\tfrac{q-1}{k}\} $$ 
	except for $k=q-1$, that is, except for the union of directed $p$-cycles.
	
	In Section \ref{sec: Waring} we consider weak Waring numbers over finite fields. 
	By applying Theorem~\ref{thm: real spec}, in Theorem \ref{thm: wkq} we show that the weak Waring number $w(k,q)$ exists if and only if the Waring number $g(k,q)$ exists (which in turn happens if and only if the graph $\G(k,q)$ is connected) and in this case, we have 
	\begin{equation} \label{eq: wkq reduction}
		w(k,q) = \begin{cases}
			g(k,q) & \qquad \text{if $\G(k,q)$ is undirected,} \\[1mm]
			g(\frac k2,q) & \qquad \text{if $\G(k,q)$ is directed,}
		\end{cases} 
	\end{equation}
	reobtaining a result in \cite{CiCo}, in a simpler way using graphs. 
	In this way, the study of weak Waring numbers reduces to the study of Waring numbers in a simple way, in particular solving the general problem.
	Additionally, in Corollary \ref{coro: wkg=gkq} we show that 
	$w(k,q)$ is the diameter of $W(k,q)=Cay(\ff_q, \check R_k)$, the symmetrized graph of $\G(k,q)=Cay(\ff_q,R_k)$, in full analogy with the known result that $g(k,q)$ is the diameter of $\G(k,q)$.
	In symbols, 
	$$ w(k,q) = {\rm diam} \big( Cay(\ff_q, R_k \cup (-R_k)) \big). $$
	Then, using \eqref{eq: wkq reduction} and a reduction formula for Waring numbers, in Theorem~\ref{thm: reduction wkq} we obtain the following reduction formula for weak Waring numbers
	
	$$ w(\tfrac{p^{ab}-1}{bc},p^{ab}) = b w(\tfrac{p^a-1}{c}, p^{a})$$  
	for certain positive integers $a$ and $b$.
	
	Finally, in Section \ref{sec: examples} we consider GP-graphs over the finite fields $\ff_{5^2}, \ff_{7^2}, \ff_{3^4}$ and $\ff_{2^8}$. 
	We give all the GP-graphs and their description as known graphs when possible. We show which has integral, real or complex spectrum. For the connected GP-graphs $\G(k,q)$ considered, we compute all associated weak Waring numbers $w(k,q)$.

	\section{The nature of the spectrum of some Cayley graphs} \label{sec:2a}
	If $G$ is an undirected graph, its adjacency matrix is symmetric and so its spectrum is real.
	In general, for an arbitrary graph $G$, we have that $Spec(G) \subset \mathbb{C}$. 
	The problem to decide if a directed graph has non-real spectrum is, in general, a difficult and elusive problem (see for instance the 2010's survey \textit{Spectra of digraphs} by Brualdi \cite{Br}). 
	
	On the other hand, we recall that there are no graphs having rational non-integral spectrum. The adjacency matrix $A$ of a graph $G$ is integral (only have $0$'s and $1$'s). Thus, the characteristic polynomial of $A$ is monic with integral coefficients and the eigenvalues are its zeros. Thus, the eigenvalues $\lambda$ are algebraic integers implying that if $\lambda \in \mathbb{Q}$ then $\lambda \in \Z$. 
	
	Hence, given an arbitrary graph, to study the \textit{nature of the spectrum} is to decide whether it is real or not; and, if real, whether it is integral or not (we will call this the \textit{integrality problem}). 
	Classifying graphs with integral spectrum is, in general, a difficult open problem (see \cite{HS}).

	Here, we will first study the real versus complex nature of some general Cayley graphs (in the next section we will focus on the subclass of GP-graphs). 
	Given a group $G$ and a subset $S$ of $G$, the \textit{Cayley graph} $X(G,S)$ is the directed graph having vertex set $G$ and where there is an oriented edge (arc) from $x$ to $y$, denoted $\vec{xy}$, if and only if $yx^{-1}\in S$. 
	It is customary to assume that $e\notin S$ so that $G$ has no loops, where $e$ is the identity in $G$. 
	Also, one assumes that $S\ne \varnothing$, since $X(G,\varnothing)$ is the empty graph with $|G|$ vertices and no edges. 
	
	If $S$ is symmetric, that is closed by inversion $S=S^{-1}$ (for instance if $S$ is a subgroup), $\vec{xy}$ is an arc if and only if $\vec{yx}$ is an arc. In this case, $xy$ is usually considered as an undirected edge instead of a double bi-oriented arc, and hence $X(G,S)$ is undirected. When $G$ is abelian, we write $0$ for $e$, $-S$ for $S^{-1}$ and $y-x$ instead of $yx^{-1}$, as usual.  
	We have that 
	$$X(G,S) \text{ is undirected } \qquad \Leftrightarrow \qquad S=S^{-1}.$$ 
	Hence, $X(G,S)$ is directed if and only if $S\ne S^{-1}$.

	\subsection*{Non-mixed Cayley graphs}
	In the remainder of the section, we will study properties of directed Cayley graphs $X(G,S)$ under the condition 
	\begin{equation} \label{eq: S cap S⁻1 = void}
		S \cap S^{-1} = \varnothing 
	\end{equation}
	(such $S$ is said \textit{antisymmetric}) which will ensure that the graphs are either directed or undirected but not mixed (i.e., graphs having both directed and undirected edges). 
	\begin{lem} \label{lem: directed XGS}
		If the Cayley graph $X(G,S)$ is directed with $S\cap S^{-1} = \varnothing$, then all of its edges are directed. 
	\end{lem}	
	
	\begin{proof}
		This is clear by the previous comments.
	\end{proof}

	The following result due to Klin et al.\@ from 2004 (see Lemma 3.2 in \cite{KMMZ})
	gives a condition for a regular directed graph to have at least one complex non-real eigenvalue, hence answering Brualdi's question in this case. We give a proof for completeness.

	\begin{lem}[\cite{KMMZ}, Lemma 3.2] \label{lem KMMZ}
		Let $\G$ be a non-empty directed regular graph without undirected edges. 
		Then, $\G$ has at least one non-real eigenvalue.
	\end{lem}
	
	\begin{proof}
		The empty graph with $m$-vertices has real spectrum $\{[0]^m\}$. 
		Hence, let $\G$ be a non-empty $n$-regular directed graph without undirected edges and suppose that 
		$Spec (\G) \subseteq \mathbb{R}$.
		Then, $\G$ has no closed directed walks of length $2$. This implies that the diagonal of its adjacency matrix 
		$A$ is zero, and hence $\Tr(A^{2})=0$.
		Also, $Spec( A^{2}) = \{ \lambda^{2}: \lambda \in Spec(\G) \}$. 
		Since $\lambda \in \mathbb{R}$ for all $\lambda\in Spec(\G)$ and the regularity degree $n$ is the biggest eigenvalue (which is positive), we have that
		$$0=\Tr(A^{2})=\sum_{\lambda \in Spec(\G)} \lambda^{2}>0,$$
		which is absurd. 
		Therefore, $\G$ has at least one non-real eigenvalue, as asserted. 
	\end{proof}

	Due to the comments at the beginning of the section and the previous lemma, we will make use of the following notational and terminological convention.
	\begin{notation}
		From now on, when we write 
		$$Spec(\G) \subset \C,$$ 
		we understand that the graph $\G$ has at least one complex non-real eigenvalue and we say that $\G$ \textit{has complex spectrum}.
		Also, when we write $Spec(\G) \subset \R$ we understand that $\G$ has real (non-integral) spectrum.	
	\end{notation}

	We have the following necessary condition on $S$ for a Cayley graph $X(G,S)$ to have complex spectrum.
	\begin{prop} \label{thm: XGS real spec}
		If the Cayley graph $X(G,S)$ is directed with $S \cap S^{-1}=\varnothing$, then $Spec(X(G,S)) \subset \mathbb{C}$. 
	\end{prop}

	\begin{proof}
		Since $S \cap S^{-1}=\varnothing$, by Lemma \ref{lem: directed XGS} all the edges of $X(G,S)$ are directed. Since $X(G,S)$ is $|S|$-regular we can apply Lemma \ref{lem KMMZ}, and therefore $X(G,S)$ has at least one non-real eigenvalue, as asserted.
	\end{proof}

	A case which is of major interest to us (here and in other works) is when $G$ is a finite commutative ring $R$ (in particular a finite field $\ff_q$). In this case, $G$ is abelian and the condition $S=-S$ holds if and only if $-1\in S$.  
	
	For a group $G$ and a subset $S \subset G$ which is not symmetric, we can consider the \textit{symmetrization} of $S$,
	\begin{equation} \label{eq: symmetrized S}
		\check S = S \cup S^{-1},
	\end{equation}
	a symmetric subset of $G$ containing $S$ (clearly, if $S$ is symmetric, then $\check S=S$). 
	We have that $X(G,S)$ is directed and $X(G,\check S)$ is undirected. 
	There is the following neat relation between these graphs and their spectra.

	\begin{thm} \label{thm: XGSy-S}
		Consider the directed Cayley graph $X(G,S)$ with $S\cap S^{-1}=\varnothing$ and let $\check S = S \cup S^{-1}$. 
		Then, the graph $X(G,\check S)$ is undirected and $2|S|$-regular while the graphs $X(G,S)$ and $X(G,S^{-1})$ are directed and $|S|$-regular. 
		Also, we have the graph union decompositions 
		\begin{equation} \label{eq: XGS u XGS^-1}
			X(G,\check S) = X(G,S) \cup X(G,S^{-1}) = \overset{_{\rightarrow}}{X}(G,S) \cup \overset{_{\leftarrow}}{X}(G,S), 
		\end{equation}
		where $\overset{_{\rightarrow}}{X}(G,S)$ denotes $X(G,S)$ with the given orientation and $\overset{_{\leftarrow}}{X}(G,S)$ denotes $X(G,S)$ with the reverse orientation. 
		Moreover, if $G$ is abelian the spectra of these graphs satisfy 
		\begin{equation} \label{eq: Spec XGS*}
			Spec(X(G,\check S)) = 2 \, \mathrm{Re} (Spec(X(G,S))). 
		\end{equation}
	\end{thm}

	\begin{proof}
		The graph $X(G,\check S)$ is undirected since $\check S$ is symmetric and $2|S|$-regular since 
		$$|\check S|=|S|+|S^{-1}|=2|S|,$$ 
		where we have used the hypothesis $S\cap S^{-1}=\varnothing$ and that $|S|=|S^{-1}|$, the inversion being a bijection from $S$ to $S^{-1}$.   
		Also, since $S$ and $S^{-1}$ are not symmetric we have that $X(G,S)$ and $X(G,S^{-1})$ are directed $|S|$-regular graphs without undirected edges, by Lemma \ref{lem: directed XGS}.
		
		The decompositions of $X(G,\check S)$ given in \eqref{eq: XGS u XGS^-1} are clear from the definitions.
		In fact, the first equality follows from the identity 
		$$X(G, S\cup T) = X(G,S) \cup X(G,T)$$ 
		for every pair of disjoint subsets $S,T$ of $G$. For the second one, notice that 
		$$\vec{xy} \in E(X(G,S)) \: \Leftrightarrow \: yx^{-1} \in S \: \Leftrightarrow \: (yx^{-1})^{-1} \in S^{-1} 
		\: \Leftrightarrow \: \vec{yx} \in E(X(G,S^{-1})).$$
		
		With respect to the spectrum, it is known that the eigenvalues of a Cayley graph $X(G,T)$ are given by 
		\begin{equation} \label{eq: chi(S)}
			\lambda_{\chi} = \chi(T) = \sum_{g \in T} \chi(g)
		\end{equation}
		where $\chi: G \rightarrow \mathbb{S}^1 \subset \C^*$ runs over the set $\hat G$ of (irreducible) characters of $G$, for $G$ abelian (see \cite{LZ2}).
		Thus, for any character $\chi$ of $G$, since $S$ and $S^{-1}$ are disjoint, we have that 
		\begin{equation} \label{eq: chi2Re}
			\begin{aligned}
				\chi(\check S) & =  \sum_{g \in S\cup S^{-1}} \chi (g) = \sum_{g \in S} \chi (g) + \sum_{g \in S^{-1}} \chi (g)  \\ & = \chi(S) + \chi(-S) = \chi(S) + \overline{\chi(S)} =  2\, \mathrm{Re}(\chi(S)),	
			\end{aligned}	
		\end{equation}
		where we have used that $\chi(-g) = \chi^{-1}(g) = \overline{\chi(g)}$, and this implies \eqref{eq: Spec XGS*}. 
	\end{proof}
	
	We point out that the hypothesis of the antisymmetry of $S$ in the theorem is necessary, as we can see in the next example borrowed from \cite{ChP}.
	
	\begin{exam}
		Consider the Cayley graphs $\G_1=X(G_1,S_1)$ and $\G_2 = X(G_2,S_2)$ where  
		$G_1=\Z_{16}$ and $G_2=\Z_4\times\Z_4$ with 
		\begin{gather*}
			S_1=\{1,2,4,5,9,10,12,13\} \subset \Z_{16}, \\ 
			S_2=\{(0,1),(0,2),(1,0),(1,2),(2,1),(2,2),(3,1),(3,3)\} \subset \Z_4 \times \Z_4.
		\end{gather*}
		These graphs are connected 8-regular graphs of 16 vertices without loops. 
		Since $S_1$ and $S_2$ are not symmetric, the graphs $\G_1$ and $\G_2$ are directed. 
		However, since 
		$$S_1 \cap (-S_1) = \{4,12\} \ne \varnothing \qquad \text{and} \qquad S_2 \cap (-S_2) = \{(0,2),(2,2)\} \ne \varnothing,$$ 
		they are mixed graphs. 
		
		In Example 6.6 of \cite{ChP}, it is proved that $\G_1$ and $\G_2$ are non-isomorphic isospectral Cayley graphs, with spectrum given by
		$$Spec(\G_i) = \{[8]^1,[4i]^1,[-2+2i]^2,[0]^9,[-2-2i]^2,[-4i]^1\} \subset \Z[i] $$ 
		for $i=1,2$.
		It is easy to see that the symmetrizations of $S_1$ and $S_2$ are
		$\check S_1 = \Z_{16} \smallsetminus \{0,8\}$ and $\check S_2 = \Z_4 \times \Z_4 \smallsetminus \{(0,0), (2,0)\}$.
		One can check that, for $i=1,2$, we have that 
		$$ Spec(\check \G_i) = \{ [14]^1, [0]^8, [-2]^7 \} $$
		with $\check \G_i=X(G_i, \check S_i)$, which is different from 
		$$2 {\rm Re}(Spec(\G_i)) = \{ [16]^1, [0]^{11}, [-4]^4 \} .$$
		Thus, equation \eqref{eq: Spec XGS*} does not hold for this mixed graphs.
		\hfill $\diamond$
	\end{exam}

	\section{The nature of the spectrum of the GP-graphs $\G(k,q)$}  \label{sec:2}
	From now on, we focus on a particular family of Cayley graphs $X(G,S)$, those with $G=\ff_q$ and $S=R_k=\{x^k: x\in \ff_q^*\}$, that is the generalized Paley graphs over finite fields which we have denoted $\G(k,q)$ in \eqref{eq: Gkq}.
	For these GP-graphs, the problem of determining the nature of the spectrum turns up to be extremely simple as we next show.

	\subsection{The nature of $Spec(\G(k,q))$ via directedness}
	First, as a consequence of some general results, we will show that a GP-graph has real spectrum if and only if it is undirected. 
	
	We begin by showing that in the directed case, a GP-graph has no undirected edges; i.e., GP-graphs are no mixed graphs. 
	\begin{lem} \label{lem: no edges}
		If the GP-graph $\G(k,q)$ is directed then all of its edges are directed; that is, $\G(k,q)$ is oriented.
	\end{lem}	
	
	\begin{proof}
		Since $\G(k,q)$ is directed, we have that $R_k \ne -R_k$ and, in particular, $-1\not \in R_k$ (for if not, $R_k=-R_k$). 
		Moreover, 
		\begin{equation} \label{eq: Rky-Rk} 
			R_k \cap (-R_k) = \varnothing, 
		\end{equation}
		since $R_k$ is a subgroup of the cyclic group $\ff_{q}^*$ and $-R_k$ is the left coset of $R_k$  containing $-1$. Hence, we are under the hypothesis of Lemma \ref{lem: directed XGS} and the result follows from it.
	\end{proof}

	We are now in a position to give the characterization of GP-graphs with real (resp.\@ complex) spectrum.
	\begin{thm} \label{thm: real spec}
		The GP-graph $\G(k,q)$ is undirected if and only if $Spec(\G(k,q)) \subset \mathbb{R}$. 
		Equivalently, $\G(k,q)$ is directed if and only if $Spec(\G(k,q)) \subset \mathbb{C}$.
	\end{thm}

	\begin{proof}
		Clearly, if $\G(k,q)$ is undirected, then its spectrum is real since its adjacency matrix is symmetric.
		Conversely, if $\G(k,q)$ is directed, by \eqref{eq: Rky-Rk} and Proposition \ref{thm: XGS real spec}
		we have that $Spec(\G(k,q))$ is complex.
	\end{proof}

	Now, we summarize the nature of the eigenvalues of the GP-graphs in terms of simple arithmetic conditions and through directedness.
	
	\begin{rem} \label{rem: nature}
		Given a GP-graph $\G=\G(k,q)$ with $k\mid q-1$, we have characterized when its spectrum is real or complex, and previously in \cite{PV3},  when it is integral. Namely, $Spec(\G)$ is real if and only if $\G$ is undirected (i.e., $q$ is even or $q$ is odd and $k\mid \frac{q-1}2$). Furthermore, if $k$ also divides $\frac{q-1}{p-1}$ then $Spec(\G)$ is integral (see \cite{PV3}). 
		
		Summing up, we have that binary GP-graphs (i.e., defined over $\ff_{2^m}$) are always integral (and undirected), that is
		\begin{equation} \label{eq: 2^m integral}
			Spec(\G(k,2^m)) \subset \Z
		\end{equation}
		for every $m$ and every $k\mid 2^m-1$. For any $q$ odd and any $k \mid q-1$ one has
		\begin{equation} \label{eq: nature}
			\begin{aligned}
				Spec(\G(k,q)) \subset \R \qquad & \Leftrightarrow \qquad k\mid \tfrac{q-1}2, \\[1mm] 
				Spec(\G(k,q)) \subset \Z \qquad & \Leftrightarrow \qquad k \mid \tfrac{q-1}{p-1}. 
			\end{aligned}
		\end{equation}
		The integrality problem will be treated in more detail in the next section.
		Also, by emphasis, we explicitly rewrite 
		\begin{gather} \label{eq: nature2}
			\begin{aligned}
				Spec(\G(k,q)) \subset  \C \qquad & \Leftrightarrow \qquad \G(k,q) \text{ is directed}, \\[1mm] 
				Spec(\G(k,q)) \subset  \R \qquad & \Leftrightarrow \qquad \G(k,q) \text{ is undirected}. 
			\end{aligned}
		\end{gather}
		In this way we have fully characterized the nature of the eigenvalues of a GP-graph.
	\end{rem}

	The study of the nature of the spectrum can be made more interesting if we consider other fields and rings besides $\C, \R$ and $\Z$ respectively, as we next make precise.
	\begin{rem}
		In \cite{PV7} we showed that the non-principal eigenvalues of $\G(k,q)$, i.e.\@ those different from $n=\frac{q-1}k$, are given by the Gaussian periods 
		$$ \eta_i^{(k,q)} = \sum_{x \in C_i^{(k,q)}} \zeta_p^{\Tr(x)} \in \mathbb{Q}(\zeta_p), \qquad 0\le i \le k-1,$$
		where $\zeta_p=e^{\frac{2\pi i}p}$ is the primitive $p$-th root of unity, $C_i^{(k,q)} = \omega^i \langle \omega^k \rangle$ is the coset 
		in $\ff_q^*$ with $\omega$ a generator of $\ff_q^*$ and $K=\mathbb{Q}(\zeta_p)$ is the cyclotomic field, the smallest number field containing $\mathbb{Q}$ and $\zeta_p$. 
		Thus, by a previous comment on rational eigenvalues at the beginning of Section~\ref{sec:2a}, we have that the spectrum of $\G(k,q)$ is contained in the ring of integers $\mathcal{O}_K$ of $K$, which equals $\Z[\zeta_p]$, that is
		$$ Spec(\G(k,q)) \subset \Z[\zeta_p].$$
		This ring has non-empty intersection with $\R$ and $\C$.
		Of course, the spectrum can actually be contained in smaller rings. 
		
		The problems to determine the smallest ring/field where the spectrum of a fixed $\G(k,q)$ lies (or in general all possible smallest rings for all the GP-graphs) and to characterize all GP-graphs with a fixed smallest ring are deeper and harder questions to consider under the name `nature of the spectrum'. 
	\end{rem}
	
	In this generality, the nature of the spectrum of GP-graphs over different fields and rings, as well the nature of the spectrum of mixed Cayley graphs (not necessarily directed or undirected) over finite rings (not necessarily finite fields) are much more challenging. 
	The complications of these tasks exceeds the scope of the present work.

	\subsection{The nature of $Spec(\G(k,q))$ arithmetically} \label{sec:2b}
	In the previous subsection we showed that the undirected (resp.\@ directed) GP-graphs $\G(k,q)$ are exactly those with real (resp.\@ complex) spectrum. We now go a step forward by giving this classification in terms of divisors of $k$ and $q-1$.	
	
	We first give a kind of analogous of Theorem \ref{thm: XGSy-S} in the particular case when  $X(G,S)$ is $X(\ff_q,(\ff_q^*)^k)=\G(k,q)$.	
	We will need the 2-adic valuation of $n$, denoted by $v_2(n)$; that is $v_2(n)=t$ if and only if $n=2^t m$ with $m$ odd.
	\begin{thm} \label{thm: nature spec v2}
		Let $q=p^m$, with $p$ an odd prime and $m\in \N$, and let $k \in \N$ such that $k\mid q-1$. 
		The graph $\G(k,q)$ is undirected if and only if $v_2(k)<v_2(q-1)$ or $v_2(k)=0$ (i.e.\@ $k$ is odd). Equivalently,  
		The graph $\G(k,q)$ is directed if and only if 
		\begin{equation} \label{eq: v2k} 
			v_{2}(k) = v_{2}(q-1) >0.
		\end{equation}
		In this case, $k$ is even, $\G(\frac{k}{2},q)$ is undirected and satisfies the relation
		\begin{equation} \label{eq: G dirs}
			\G(\tfrac k2,q) = \overset{_{\rightarrow}}{\G}(k,q) \cup \overset{_{\leftarrow}}{\G}(k,q),
		\end{equation} 
		where $\overset{_{\rightarrow}}{\G}(k,q)$ denotes the directed graph $\G(k,q)$ with the given orientation and $\overset{_{\leftarrow}}{\G}(k,q)$ denotes $\G(k,q)$ with the reverse orientation. Also, their spectra satisfy the relation
		\begin{equation} \label{eq: Spec Gk2q = 2Re Spec Gkq}
			Spec(\G(\tfrac k2,q)) = 2 \, \mathrm{Re} (Spec(\G(k,q))). 
		\end{equation}
	\end{thm}
	
	\begin{note}
		$\G(\frac{k}{2},q)$ is the underlying undirected graph of the directed graph $\G(k,q)$. 	
	\end{note}

	\begin{proof}
		We know that $\G(k,q)$ is directed if $q$ is odd and $k \nmid \frac{q-1}2$ and, since $k\mid q-1$, this happens if and only if $v_2(k)=v_2(q-1)>0$ (hence $k$ is even).
		On the other hand, in the case that $\G(k,q)$ is directed, 
		since $k\mid q-1$ and $k$ is even, then $\frac{k}{2}\mid \frac{q-1}{2}$ and so $\G(\frac{k}{2},q)$ is undirected, as asserted.
		
		Now, we show the decomposition in \eqref{eq: G dirs}. 
		By \eqref{eq: Rky-Rk} we know that $R_k$ and $-R_k$ are disjoint sets, where $R_k=\{ x^{k} : x \in \ff_{q}^*\}$. By Theorem \ref{thm: XGSy-S}, it is enough to prove that the symmetrization 
		of $R_k$ equals $R_{\frac k2}$, that is
		\begin{equation} \label{eq: Rk Rk2}
			\check R_k = R_{k} \cup (-R_{k}) = R_{\frac{k}{2}}.
		\end{equation}
		Since $\frac{k}{2} \mid k$, we have that $R_{k}\subset R_{\frac{k}{2}}$. Also, since $\G(\frac{k}{2},q)$ is undirected, the set $R_{\frac{k}{2}}$ is symmetric and hence $-1 \in R_{\frac{k}{2}}$. Taking into account that the set $R_{\frac{k}{2}}$ is closed under multiplication we obtain that 
		$-R_{k} \subset R_{\frac k2}$ and, therefore,  
		$R_{k}\cup (-R_{k}) \subset R_{\frac{k}{2}}$.
		Finally, since $R_k $ and $-R_k$ are disjoint,  
		and $|R_k|=|-R_k|=\frac{q-1}k$, we have 
		$$|R_{k}\cup (-R_{k})| = 2 \tfrac{q-1}k =  |R_{\frac{k}{2}}|.$$ 
		Hence, we obtain \eqref{eq: Rk Rk2} 
		which, by Theorem \ref{thm: XGSy-S}, directly implies \eqref{eq: G dirs} and \eqref{eq: Spec Gk2q = 2Re Spec Gkq}.
	\end{proof}

	Recall that since $\G(k,q)=\G(\gcd(k,q-1),q)$, 
	the different GP-graphs over $\ff_q$ are parameterized by the divisors $k\mid q-1$. More precisely, 
	$$\G(k,q) = \G(k',q) \quad \Leftrightarrow \quad \gcd(k,q-1) = \gcd(k',q-1).$$ 
	
	The above observation and the previous proposition allow us to give the complete characterization of the nature of the spectrum of GP-graphs $\G(k,q)$ in terms of divisors of $k$ and $q-1$. 
	
	\begin{coro} \label{coro: wkqs exp}
		Let $q=p^m$ with $p$ an odd prime and $m\in \N$ and suppose that 
		$$q-1 = 2^t r$$ 
		with $t\ge 0$ and $r$ odd. Then we have: 
		\begin{enumerate}[$(a)$]
			\item The GP-graph $\G(2^t s, q)$ with $s\mid r$ has complex spectrum (i.e.\@ are directed). \msk 
			
			\item The GP-graph $\G(2^{t'} s, q)$ with $t'<t$ and $s \mid r$ has real spectrum (i.e.\@ are undirected).
		\end{enumerate}
		In particular, if $k$ is odd then $\G(k,q)$ is undirected for any $q$.
		
		Furthermore, if $r=p_1^{e_1} \cdots p_s^{e_s}$ is the prime factorization of $r$, then there are $(t+1)N$ GP-graphs $\G(k,q)$ over $\ff_q$ where 
		\begin{equation} \label{eq: NC}
			N := N_\C(q) := (e_1+1) \cdots (e_s+1),	
		\end{equation}
		out of which $N$ are directed (complex spectrum) and $tN$ are undirected (real spectrum).  
	\end{coro}

	\begin{proof}
		Since $k \mid q-1$ and $q-1 = 2^t r$ with $r$ odd, $k$ is of the form $k=2^{t'}s$ with $t'\le t$ and $s \mid r$.
		The results in ($a$) and ($b$) are hence direct consequences of Theorems \ref{thm: real spec} and \ref{thm: nature spec v2}. 
		The remaining statements are clear.
	\end{proof}
	
	We already know that for $q$ even, $\G(k,q)$ is undirected and $k$ must be odd. The previous result generalizes this to $q$ odd also. On the other hand, if $\G(k,q)$ is directed then $k$ must be even.
	
	\begin{exam} \label{exam: G union dirs} 
		The GP-graphs over $\ff_{3^5}$ are given by the divisors of $3^5-1=242=2\cdot 11^2$. Hence, we have the 6 graphs
		(notice that $t=1$ and $N=3$ in the notation of Theorem~\ref{thm: nature spec v2})
		$$\G(1,3^5), \quad \G(2,3^5), \quad \G(11,3^5), \quad \G(22,3^5), \quad \G(121,3^5), \quad \text{and} \quad \G(242,3^5),$$ 
		out of which those with $k$ even are directed. Also, by \eqref{eq: G dirs} we have 
		$$ \G(k,3^5) = \overset{_{\rightarrow}}{\G}(2k,3^5) \cup \overset{_{\leftarrow}}{\G}(2k,3^5)$$ 
		for $k=1, 11, 121$. 
		\hfill $\diamond$
	\end{exam}

	Notice that, for any $\ff_q$ there are always trivial GP-graphs $\G(k,q)$ having integral, real or complex spectrum. Namely, $\G(1,q)=K_q$, $\G(\frac{q-1}2,q)$ for $q$ odd, which is the disjoint union of $p$-cycles, and $\G(q-1,q)$, which is the disjoint union of oriented $p$-cycles, have integral, real and complex spectra, respectively (see Examples \ref{exam: complete} and  \ref{exam: cycles} below for more details).
	We remark that $\G(2,q)$ can have real or complex spectrum depending on the congruence of $q$ modulo $4$ and in the real case, it is integral for quadratic fields, i.e.\@ when $q=p^{2m}$ (see Example~\ref{exam: paley}). 
	
	Aside from these `trivial' GP-graphs, can we ensure the existence of some other families of GP-graphs with integral,
	real or complex spectrum? 
	This is the topic of the next two sections. In the next section we give some more families with integral/real/complex spectrum and in Section \ref{sec: integral}, we focus on the construction of infinite families of integral GP-graphs.

	\section{Families of GP-graphs with known spectrum} \label{sec: 4}
	Here, we recollect those families of GP-graphs with known spectra and study their nature.  
	We will use some results previously obtained in our works \cite{PV3} and \cite{PV18}. 
	We compare the known spectrum with the results in Theorems \ref{thm: real spec} and \ref{thm: nature spec v2}.
	Here, in the families, either $k$ is fixed and $q$ moves or else both $(k,q)$ move (in this case, either $k$ has some expression in terms of $q$ or $(k,q)$ form a semiprimitive pair).
	
	In all the examples of the section we set $q=p^m$ with $p$ prime and $m \in \N$, $k\mid q-1$ and $n=\frac{q-1}{k}$.
	When possible, we will also give their parameters as strongly regular graphs.
	A \textit{strongly regular graph} with parameters  $v,r,e,d$, denoted 
	$$srg(v,r,e,d),$$
	is an $r$-regular graph of order $v$ such that any pair of adjacent vertices has $e$ neighbors in common and any pair of non-adjacent vertices has $d$ common neighbors. These parameters satisfy the relation
	$$(v-r-1)d= r(r-e-1).$$
	In our case, when $\G(k,q)$ is a strongly regular graph, we will write it as 
	$$ \G(k,q)= srg(q,n,e,d),$$
	with $q=p^m$ and $n=\frac{q-1}k$.
	It is well-known that if $d\ne 0$, the graph is connected with 3 eigenvalues and has diameter 2.

	\subsection{GP-graphs with small $k$ fixed} 
	We first consider the cases of $\G(k,q)$ with small fixed $k$, that is $1\le k\le 4$ for arbitrary $q$.

	We begin with the trivial case of complete graphs.
	\begin{exam}[\textit{The graphs $\G(1,q)$}] \label{exam: complete}
		It is the complete graph 
		$$K_q=srg(q,q-1,q-2,0),$$ 
		which is connected, undirected, strongly regular with 2 different integral eigenvalues; in fact, 
		$$Spec(K_q) = \{[q-1]^1, [-1]^{q-1}\}.$$ 
		Here, the condition for integrality in \eqref{eq: nature} holds trivially. 
		\hfill $\diamond$ 
	\end{exam}

	In the next example we consider the graphs $P_q = \G(2,q)$, which are the classic (undirected) Paley graphs $\mathcal{P}_q$ or the directed Paley graphs $\vec{\mathcal{P}}_q$, depending on the congruence of $q$ modulo 4.
	
	\begin{exam}[\textit{The graphs $\G(2,q)$}] \label{exam: paley}
		We now compute the spectrum of $P_q =\G(2,q)$ by using Weil and Gauss sums. Since $P_q$ is the Cayley graph $Cay(\ff_q, (\ff_q^*)^2)$, by \eqref{eq: chi(S)}, its eigenvalues are given by the sums 
		$$ \lambda = \sum_{x \in (\ff_q^*)^2} \chi(x)$$ 
		where $\chi$ is an additive character of $\ff_q$.
		The characters of $\ff_q$ are $\{\chi_\alpha\}_{\alpha \in \ff_q}$ where for any $x\in \ff_q$ we have 
		$$\chi_\alpha (x) = \zeta_{p}^{\Tr(\alpha x)}$$ 
		with $\zeta_{p}=e^{\frac{2\pi i}{p}}$.
		Thus, for each $\alpha \in \ff_q$ we have the eigenvalue
		$$\lambda_{\alpha} = \sum_{y\in\{z^2: z\in \ff_{q}^*\}} \zeta_{p}^{\Tr(\alpha y)} = 
		\tfrac{1}{2}\sum_{x\in \ff_{q}^*}\zeta_{p}^{\Tr(\alpha x^2)} = 
		\tfrac{1}{2} \Big\{ \sum_{x\in \ff_{q}}\zeta_{p}^{\Tr(\alpha x^2)}-1 \Big\}.$$
		
		Robert Coulter showed (see Theorem 2.5 in \cite{Co}), using explicit values of quadratic Gauss sums (see Chapter 5 in \cite{LN}), that if $q=p^{m}$ with $p$ an odd prime and $\alpha\in \ff_{q}^*$ then
		$$\sum_{x\in \ff_{q}} \zeta_{p}^{\Tr(\alpha x^2)}=
		\begin{cases}
			(-1)^{m-1} \sqrt{q} \, \eta(\alpha) 		  & \qquad  \text{if $p\equiv 1 \pmod{4}$},\\[2mm]
			(-1)^{m-1} i^{m} \sqrt{q} \, \eta(\alpha) & \qquad \text{if $p\equiv 3 \pmod{4}$,}
		\end{cases} $$
		where $\eta$ is the quadratic character of $\ff_{q}$ and $\Tr:=\Tr_{\ff_q/\ff_p}$ is the trace map, and hence	
		$$\lambda_{\alpha} = \begin{cases}
			\tfrac 12 \{ (-1)^{m-1} \sqrt{q} \,\eta(\alpha)-1 \} 	     & \qquad \text{if $p\equiv 1 \pmod{4}$},\\[2mm]
			\tfrac 12 \{ (-1)^{m-1} i^{m} \sqrt{q} \, \eta(\alpha)-1 \}  & \qquad \text{if $p\equiv 3 \pmod{4}$.}
		\end{cases} $$
		
		By taking into account that there are exactly $\frac{q-1}{2}$ elements in $\ff_{q}^*$ such that $\eta(\alpha)=1$ and 
		$\frac{q-1}{2}$ elements in $\ff_{q}^*$ such that $\eta(\alpha)=-1$, we obtain that
		\begin{equation} \label{eq:SpecPq*}
			Spec(P_q) = \begin{cases}
				\{ [\tfrac{p^m-1}2]^{1}, [\tfrac{-1+\sqrt{p^m}}2]^{\frac{p^m-1}2 }, [\tfrac{-1-\sqrt{p^m}}2]^{\frac{p^m-1}2 } \} & \: \text{if $p \equiv 1 \pmod 4$}, \\[2mm]
				\{ [\tfrac{p^m-1}2]^{1}, [\tfrac{-1+i^{m}\sqrt{p^m}}2]^{\frac{p^m-1}2}, [\tfrac{-1-i^{m}\sqrt{p^m}}2]^{\frac{p^m-1}2 } \} & \: \text{if $p \equiv 3 \pmod 4$}. \end{cases}
		\end{equation}

		Since $2\mid q-1$ we have that $q$ is odd and there are two cases: ($a$) $q \equiv 1 \pmod 4$ and ($b$) $q \equiv 3 \pmod 4$.  
		
		\noindent ($a$) If $q\equiv 1 \pmod 4$, that is $p\equiv 1 \pmod 4$ or $p\equiv 3 \pmod 4$ and $m$ even, then $\G(2,q)$ is the undirected Paley graph $\mathcal{P}_q$ with real spectrum (integral if $q$ is a square), a connected strongly regular graph with 3 different eigenvalues. 
		In fact, we have 
		$$\mathcal{P}_q = srg(q, \tfrac{q-1}2, \tfrac{q-5}4, \tfrac{q-1}4)$$ 
		with 
		\begin{equation} \label{spec paley}
			Spec(\mathcal{P}(q)) = \big \{ [\tfrac{q-1}2]^1, [\tfrac{-1+(-1)^m \sqrt{q}}2]^{n}, [\tfrac{-1-(-1)^m\sqrt{q}}2]^{n} \big \}.
		\end{equation}

		\noindent ($b$)	If $q\equiv 3 \pmod 4$,  hence $p\equiv 3 \pmod 4$ and $m$ odd, then $\G(2,q)$ is the directed Paley graph $\vec{\mathcal{P}}_q$.

		It is reassuring to note that the expressions in \eqref{eq:SpecPq*} coincide with \eqref{spec paley} and the ones obtained in \cite{PV3}, 
		with complex non-real spectrum 
		\begin{equation} \label{spec paley dir}
			Spec(\vec{\mathcal{P}}(q)) = \{[\tfrac{q-1}2]^1, [\tfrac{-1+(-1)^m i^m \sqrt{q}}2]^{n}, [\tfrac{-1-(-1)^m i^m \sqrt{q}}2]^{n}\},
		\end{equation}
		where we computed the same spectrum using Gaussian periods.
		
		With respect to Corollary \ref{coro: wkqs exp}, the graph $\G(2,q)$ with $q=p^m$ and $2\mid q-1$ is directed if and only if $q-1=2r$ with $r$ odd, that is, if and only if $q\equiv 3\pmod 4$.
		\hfill $\diamond$
	\end{exam}

	Now, we study the graphs $\G(k,q)$ with $k=3,4$.
	\begin{exam}[\textit{The graphs $\G(3,q)$}] \label{exam: G3} 
		The graph $\G(3,q)$ is connected and undirected, hence with real spectrum. Its spectrum is given in Theorem 3.1 in \cite{PV3} where there are three cases: 
		\begin{enumerate}[$(a)$]
			\item $p\equiv 1\pmod 3$ with $3\mid m$, \sk 
			
			\item $p\equiv 1\pmod 3$ with $3 \nmid m$ and $m$ even, \sk 
			
			\item $p\equiv 2\pmod 3$ with $m$ even.
		\end{enumerate}
		The spectrum is integral in cases ($a$) and ($c$). In case ($c$), the graph is strongly regular with 3 different eigenvalues, while in cases ($a$) and ($b$) it has 4 different eigenvalues. 
		\hfill $\diamond$
	\end{exam}

	\begin{exam}[\textit{The graphs $\G(4,q)$}] \label{exam: G4} 
		The graph $\G(4,q)$ is connected, with the only exception of $\G(4,9)$. 
		Its spectrum is given in Theorem 3.2 in \cite{PV3} where there are five cases: 
		\begin{enumerate}[$(a)$]
			\item $p\equiv 1\pmod 4$ with $m \equiv 0 \pmod 4$, \sk 
			
			\item $p\equiv 1\pmod 4$ with $m \equiv 2 \pmod 4$, \sk 
			
			\item $p\equiv 1\pmod 4$ with $m$ and $n$ odd, \sk 
			
			\item $p\equiv 1\pmod 4$ with $m$ odd and $n$ even, and \sk 
			
			\item $p\equiv 3\pmod 4$ with $m$ even.  
		\end{enumerate}	
		The graph is undirected (hence with real spectrum) in cases ($a$), ($b$), ($d$) and ($e$), and directed (hence with complex spectrum) in case ($c$). 
		The spectrum is integral in cases ($a$) and ($e$). In case ($e$), the graph is strongly regular with 3 different eigenvalues, while in cases ($a$)--($d$) it has 5 different eigenvalues. 
		
		Relative to Corollary \ref{coro: wkqs exp}, the graph $\G(4,q)$ with $q=p^m$ and $4\mid q-1$ is directed if and only if $q-1=4r$ with $r$ odd, that is if and only if $q\equiv 5\pmod 8$, which is equivalent to the conditions $p\equiv 1 \pmod 4$ and $m$ odd given in item ($c$). For example, $\G(4,5^{2m+1})$ and $\G(4,13^{2m+1})$ are directed for every $m$.	
		\hfill $\diamond$ 
	\end{exam}

	\subsection{GP-graphs with $k$ not fixed}
	Here we take $k$ depending on $p$ and/or some other parameter. In particular, we consider GP-graphs which are cycles, Hamming GP-graphs and semiprimitive GP-graphs.

	We begin with GP-graphs which are cycles.
	\begin{exam}[\textit{Cycles}] \label{exam: cycles}
		We now consider the graphs $\G(k,q)$ with $k= \frac{q-1}2, q-1$.
		
		\noindent $(a)$ \textit{The graph $\G(\frac{q-1}2,q)$, $q$ odd}. 
		It is the disjoint union of $p^{m-1}$ copies of the undirected $p$-cycle $C_p$ (see Proposition 3.2 in \cite{PV18}), hence disconnected for all $m>1$, with real non-integral spectrum 
		\begin{equation} \label{eq: Spec Cp}
			Spec(\G(\tfrac{q-1}2,q)) = \{ [2\cos(\tfrac{2\pi j}p)]^{p^{m-1}} \}_{0 \le j \le p-1}.
		\end{equation}
		Notice that $v_2(\tfrac{q-1}2,q) < v_2(q-1)$, in accordance with Theorem \ref{thm: nature spec v2}. \sk

		\noindent $(b)$ \textit{The graph $\G(q-1,q)$, $q$ odd}. 
		It is the disjoint union of $p^{m-1}$ copies of the directed $p$-cycle $\vec{C}_p$ (see Proposition 3.2 in \cite{PV18}), thus disconnected for all $m>1$, with complex spectrum 
		\begin{equation} \label{eq: Spec dir Cp}
			Spec(\G({q-1},q)) = \{ [\zeta_p]^{p^{m-1}} \}_{0 \le j \le p-1},
		\end{equation}
		where $\zeta_p = e^{\frac{2\pi i}p}$ is the primitive $p$-th root of unity. 
		Notice that by \eqref{eq: Spec Cp} and \eqref{eq: Spec dir Cp} we check that expression \eqref{eq: Spec Gk2q = 2Re Spec Gkq} holds in this case, that is 
		$$ Spec(\G({q-1},q)) = 2Re (Spec(\G(\tfrac{q-1}2,q))).$$

		\noindent $(c)$ \textit{The graph $\G(q-1,q)$, $q$ even}. 
		The graphs $\G(2^m-1,2^m)$ with $m\in \N$ are undirected and the disjoint union of $2^{m-1}$ copies of $K_2$, hence disconnected except for $\G(1,2)=K_2$. They are the only bipartite GP-graphs (see Theorem 4.2 in \cite{PV18}), with spectrum given by 
		$$ Spec(\G(2^m-1,2^m)) = \{[1]^{2^{m-1}}, [-1]^{2^{m-1}}\}$$ 
		which is clearly integral.
		\hfill $\diamond$ 
	\end{exam}
	
	We continue with GP-graphs, which are Hamming graphs.
	\begin{exam}[\textit{Hamming GP-graphs}] 
		Connected GP-graphs $\G(k,q)$ which are Hamming graphs $H(b,q)$, were characterized by Lim and Praeger in 2009 (see \cite{LP}). In fact, 	
		\begin{equation} \label{eq: GP Hamming}
			\G(\tfrac{p^{bm}-1}{b(p^m-1)}, p^{bm}) = H(b,p^{m})
		\end{equation}	
		with $b\mid \tfrac{p^{bm}-1}{p^m-1}$.
		These graphs have integral spectrum (this is well-known, see for instance Example~4.3 in \cite{PV3}) given by 
		$$ Spec(H(b,q)) = \{ [\ell q-b]^{\binom{b}{\ell}(q-1)^{b-\ell}} \}_{0\le \ell \le b}.$$
		Clearly, $H(1,q)=K_q$. Taking $b=2$ in \eqref{eq: GP Hamming} we get the lattice (or rook's graph) which is a strongly regular graph 
		$$ \G(\tfrac{q+1}2,q^2) = H(2,q) = L_{q,q} = srg(q^2,2(q-1), q-2,2).$$
		In fact, this is the only Hamming GP-graph which is strongly regular since it is the only one having exactly 3 different eigenvalues (see Section \ref{sec:3}). 
		The graph $H(3,q)$ is uniquely determined by its spectrum when $q\ge 36$ (\cite{BvDK}).
		\hfill $\diamond$	
	\end{exam}
	
	Finally, we give a big and important family of GP-graphs, the semiprimitive ones.
	\begin{exam}[\textit{Semiprimitive GP-graphs}] \label{exam: semip}
		A graph $\G(k, q)$ with $q=p^m$ with $p$ prime and $m\in \N$ is \textit{semiprimitive} if $k=2$ and $q\equiv 1 \pmod 4$, i.e.\@ it is the classic Paley graph $\vec{\mathcal{P}}_q=\G(2,q)$, or else 
		$$\text{$k>2$ \quad with \quad $k \mid p^t+1$ \quad for some \quad $t \mid \tfrac m2$ 
			\quad and \quad $k \ne p^{\frac m2}+1$}$$ 
		(see Definition 5.1 in \cite{PV3}).
		Hence, every semiprimitive GP-graph $\G(k,q)$ is undirected and connected. 
		
		In \cite{PV3}, we have studied the spectrum and some properties of the semiprimitive GP-graphs. 
		For instance, every semiprimitive GP-graph $\G(k,q)$ is a strongly regular graph 
		$srg(q,n,e,d)$ 
		with 3 different integer eigenvalues. Namely, we have that (see Theorem~5.4 in \cite{PV3})
		\begin{equation} \label{eq: Spec semipGP}
			Spec(\G(k,q)) = \{ [n]^1, [\tfrac{\sigma(k-1)\sqrt q-1}k]^n, [-\tfrac{\sigma \sqrt q +1}k]^{(k-1)n} \}
		\end{equation}
		where $n=\frac{q-1}k$, $m=2ts$, $t$ is the least integer $j$ such that $k \mid p^j+1$ and $\sigma= (-1)^{s+1}$.
		
		The parameters $(q,n,e,d)$ are $q=p^m$, with $m=2ts$, $s\in \N$ and $t$ is the minimum positive integer such that $k\mid p^t+1$, $n=\frac{q-1}k$ and where
		\begin{align*}
			& d = n + (\sigma \sqrt q +\lambda)\lambda =  \lambda^2 + \sigma \sqrt q \lambda +n, \\
			& e = d + \sigma \sqrt q + 2\lambda = \lambda^2 +(\sigma \sqrt q +2)\lambda +\sigma\sqrt q +n, 
		\end{align*}
		are quadratic expressions in one of the non-principal eigenvalues (see \eqref{eq: Spec semipGP})
		$$\lambda= -\tfrac{\sigma \sqrt q +1}{k}$$
		(see Theorem 5.8 in \cite{PV3}, where also the parameters of $\G(k,q)$ as a pseudo-Latin square graph and as a distance regular graph are given). 
		
		Summing up, semiprimitive GP-graphs are integral strongly regular graphs.
		\hfill $\diamond$
	\end{exam}
	
	For more details on all the previous examples we refer to Examples 2.3--2.4, Theorems~3.1--3.2, Examples 4.2--4.3, Theorems~5.4 and 5.8 in \cite{PV3}, and Proposition 3.2 and Theorem~4.2 in \cite{PV18}. 
	For instance, the description of the spectra of $\G(3,q)$ and $\G(4,q)$ are more involved and are not written down explicitly in Examples \ref{exam: G3} and \ref{exam: G4}.

	\section{Integral spectrum} \label{sec: integral}
	Now, we study in more detail the integrality problem. 
	We know from \eqref{eq: nature} that $\G(k,q)$ has integral spectrum if and only if 
	\begin{equation*} \label{eq: k|(q-1)/(p-1)}
		k \mid \tfrac{q-1}{p-1}.
	\end{equation*}
	Hence, any GP-graph over a field of characteristic 2 is integral, i.e.\@ $Spec(\G(k, 2^m))~\subset~\Z$, and thus one can assume that $q$ is odd when studying integrality of GP-graphs.
	Furthermore, for $q$ odd, we have that 
	$$\G(\tfrac{q-1}{p-1},q)$$ 
	and all of its GP-supergraphs (that is those $\G(k,q)$ with $k \mid \tfrac{q-1}{p-1}$) have integral spectrum. 
	
	Conversely, they are all the integral GP-graphs over $\ff_q$.
	In fact, putting condition \eqref{eq: nature} in terms of $p$-adic valuations, we readily obtain the following result characterizing integral spectrum arithmetically, which complements Corollary \ref{coro: wkqs exp}.
	
	\begin{coro} \label{coro: integral GPs}
		Let $q=p^m$ with $p$ a prime and $m\in \N$, let  
		$\tfrac{q-1}{p-1} = p_1^{f_1} \cdots p_s^{f_s}$
		be the prime factorization of $\frac{q-1}{p-1}$ and let $k\mid q-1$. 
		Then, 
		\begin{equation}  \label{eq: condition Gkq integral}
			Spec(\G(k, q)) \subset \Z \qquad \Leftrightarrow \qquad 
			v_{p_i}(k) \le f_i  \quad (1\le i \le s).
		\end{equation} 
		Moreover, there are 
		\begin{equation} \label{eq: NZ}
			N_\Z	(q) := (f_1+1) \cdots (f_s+1)
		\end{equation}
		integral GP-graphs over $\ff_q$.	
	\end{coro}
	
	\begin{proof}
		Immediately from the above comments, condition (\ref{eq: nature}) and Theorem \ref{thm: nature spec v2}. 
	\end{proof}

	\begin{rem} 
		For any fixed $q$, we let $\G(q)$ be the set of GP-graphs $\G(k,q)$ over $\ff_q$ and we denote by $\G_\C(q)$, $\G_\R(q)$, 
		$\G_{\R \smallsetminus \Z}(q)$ and $\G_\Z(q)$ the set of GP-graphs defined over $\ff_q$ which has complex (non-real) spectrum, real spectrum, real non-integral spectrum and integral spectrum, respectively.   
		
		The number of these sets is given by 
		$$ \#\G(q) = \sigma(q-1),$$ 
		where $\sigma(n)$ is the number of divisors of $n$ and 
		\begin{equation} \label{eq: NUmber of Gkqs}
			\begin{aligned}
				&	\#\G_\C(q) = N_\C(q), \\[1mm]
				&	\#\G_\R(q) = v_2(q-1) N_\C(q), \\[1mm]
				&	\#\G_\Z(q) = N_\Z(q), 
			\end{aligned}
		\end{equation}
		where $N_\C(q)$ and $N_\Z(q)$ are the numbers given in \eqref{eq: NC} and \eqref{eq: NZ}, respectively. From this, we clearly obtain
		\begin{equation} \label{eq: Gkq R-Z}
			\G_{\R \smallsetminus \Z}(q) = \sigma(q-1) - (N_\C(q) + N_\Z(q)) = v_2(q-1) N_\C(q) - N_\Z(q).    
		\end{equation}
	\end{rem}

	In what follows we will give sufficient conditions to get families of integral GP-graphs.

	\subsection{Basic arithmetic criteria} 
	We begin by giving some basic arithmetic criteria that ensure the integrality of the spectrum of a GP-graph.
	
	\begin{lem} \label{prop: criteria}
		Consider $\G(k,q)$ with $q=p^m$, $p$ prime, $k\mid q-1$ and $m\in\N$.
		\begin{enumerate}[$(a)$]
			\item If $\gcd(k,p-1)=1$, then $Spec(\G(k,q))\subset \Z$. \msk 
			
			\item If $p\equiv 1 \pmod k$ and $k\mid m$ then $Spec(\G(k,q))\subset \Z$. \msk
			
			\item If $p\equiv -1 \pmod k$ and $m$ is even then $Spec(\G(k,q))\subset \Z$. 
			
		\end{enumerate} 	
	\end{lem}	
	
	\begin{proof}
		($a$) If $\gcd(k,p-1)=1$, since $q-1 = \frac{q-1}{p-1} (p-1)$ and $k\mid q-1$ but $k\nmid p-1$, then $k\mid \frac{q-1}{p-1}$, and this implies the integrality of the spectrum.
		
		\noindent ($b$) 
		First note that 
		$$\tfrac{q-1}{p-1} = \tfrac{p^m-1}{p-1} = p^{m-1} + \cdots + p^2+p+1.$$ 
		If $p\equiv 1 \pmod k$ and $k\mid m$ we have $\frac{q-1}{p-1} \equiv m \equiv 0 \pmod k$, and hence $k\mid \frac{q-1}{p-1}$, which implies the integrality of the spectrum of $\G(k,q)$.
		
		\noindent ($c$) 
		Similarly, we have $\frac{q-1}{p-1} \equiv 1+(-1) + \cdots +1+(-1) \equiv 0 \pmod k$ if $m$ is even and again we have $k\mid \frac{q-1}{p-1}$, hence the spectrum of $\G(k,q)$ is integral.
	\end{proof}

	\subsubsection*{Infinite families of GP-graphs}	
	Next, in a series of propositions and corollaries, we will give infinite families of integral GP-graphs.
	We begin with the following one which is immediate from items ($b$) and ($c$) of Lemma \ref{prop: criteria}. 
	
	\begin{prop} \label{prop: finite families of integral GPs}
		Let $p$ be a prime. We have the families of integral GP-graphs:
		\begin{enumerate}[$(a)$]
			\item $\{\G(k, p^{kt})\}_{t\in \mathbb{N}}$, for any $k\in \mathbb{N}$ with $k\mid p-1$. \msk
			
			\item $\{\G(k, p^{2t})\}_{t\in \mathbb{N}}$, for any $k\in \mathbb{N}$ with $k\mid p+1$. 
		\end{enumerate} 	
	\end{prop}

	Notice that $\G(k, p^{2t})$ is a particular case of a semiprimitive GP-graph when $t>2$.
	Also, for $k=2$, both items ensure that $\G(2,p^{2t})$ is an integral GP-graph. This is known, being Paley graphs over quadratic fields, and thus the corollary generalizes this fact.
	
	Here we illustrate the previous proposition for the first primes.
	\begin{exam} 
		We consider the cases $p=3,5,7,11$. We will use Proposition \ref{prop: finite families of integral GPs} and we exclude the case $k=1$ because it is trivial.
		
		\noindent $(a)$
		If $p=3$, we have the families  
		$$ \{ \G(2,3^{2t})\}_{t \in \N} \qquad \text{and} \qquad \{\G(4,3^{2t})\}_{t \in \N}$$ 
		of integral GP-graphs. In particular, $\G(2,9)$, $\G(2,81)$ and $\G(4,9)$, $\G(4,81)$ are integral. 
		
		\noindent $(b)$
		If $p=5$, by looking at the divisors of $4$ and $6$ we have the families 
		$$ \{ \G(2,5^{2t})\}_{t \in \N}, \quad \{\G(4,5^{4t})\}_{t \in \N}, \quad \text{and} \quad \{\G(3,5^{2t})\}_{t \in \N},  \quad \{\G(6,5^{2t})\}_{t \in \N},$$ 
		of integral GP-graphs. 
		In particular, $\G(2,25)$, $\G(3,25)$, $\G(6,25)$ and $\G(2,625)$, $\G(3,625)$, $\G(4,625)$, $\G(6,625)$ are integral. 
		
		\noindent $(c)$
		If $p=7$, by looking at the divisors of $6$ and $8$ we have the families  
		\begin{gather*}
			\{\G(2,7^{2t})\}_{t \in \N}, \quad \{\G(3,7^{3t})\}_{t \in \N}, \quad \{\G(6,7^{6t})\}_{t \in \N}, \\ 
			\{\G(4,7^{2t})\}_{t \in \N}, \quad \{\G(8,7^{2t})\}_{t \in \N},	
		\end{gather*}
		of integral GP-graphs. 
		In particular, $\G(2,49)$, $\G(4,49)$, $\G(8,49)$ and $\G(3, 343)$ are integral. 
		Also, $\G(2,7^6)$, $\G(3,7^6)$, $\G(4,7^6)$, $\G(6,7^6)$ and $\G(8,7^6)$ are integral.

		\noindent $(d)$
		If $p=11$, by looking at the divisors of $10$ and $12$ we have the families  
		\begin{gather*}
			\{\G(2,11^{2t})\}_{t \in \N}, \quad \{\G(5,11^{5t})\}_{t \in \N}, \quad \{\G(10,11^{10t})\}_{t \in \N}, \\ 
			\{\G(3,11^{2t})\}_{t \in \N}, \quad \{\G(4,11^{2t})\}_{t \in \N}, \quad \{ \G(6,11^{2t})\}_{t \in \N}, \quad \{\G(12,11^{2t})\}_{t \in \N},
		\end{gather*}
		of integral GP-graphs. 
		In particular, $\G(k,11^2)$ for $k=2,3,4,6,12$ are integral. 
		Also, $\G(5,11^5)$ and $\G(k,11^{10})$ with $k=2,3,4,5,6,10,12$ are integral.  
		\hfill $\diamond$
	\end{exam}

	As a direct application of Proposition \ref{prop: finite families of integral GPs}, a more general family of examples is given in the next statement, for prime numbers $p$ of the form $r^\ell \pm 1$, with $r$ another prime.
	\begin{coro} \label{coro: cadenitas}
		Let $p$ be a prime. We have the following infinite families of integral GP-graphs:
		\begin{enumerate}[$(a)$]
			\item 	$\{\G(r,p^{rt}), \G(r^2,p^{r^2t}), \ldots, \G(r^\ell,p^{r^\ell t}) \}_{t\in \N}$, if $p=r^\ell +1$ with $r$ prime and $\ell \in \N$. \msk
			
			\item 	$\{\G(r,p^{2t}), \G(r^2,p^{2t}), \ldots, \G(r^\ell,p^{2t}) \}_{t\in \N}$, if $p=r^\ell -1$ with $r$ prime and $\ell \in \N$.
			
		\end{enumerate}
		
	\end{coro}

	In the previous proposition, for any given prime number we show that we can obtain a finite number of families of integral GP-graphs. Now, we will be able to show the existence of an infinite number of families of integral GP-graphs obtained from a single prime.
	
	\begin{prop} \label{prop: families of integral GPs}
		For any prime $p$, we have the infinite family of integral GP-graphs
		$$ \{\G(k,p^{\varphi(k)t})\}_{t\in \mathbb{N}}, $$ 
		for any $k\in \mathbb{N}$ odd with $\gcd(k,p(p-1))=1$, where $\varphi$ is the Euler totient function. \msk 
	\end{prop}
	
	\begin{proof}
		The condition $\gcd(k,p(p-1))=1$ is equivalent to $\gcd(k,p-1)=\gcd(k,p)=1$. In particular, $k$ must be odd. Thus, by the Euler-Fermat theorem we have that
		$$ p^{\varphi(k)t}\equiv 1\pmod{k} $$
		for every $t \in \N$. 
		Hence, $k\mid p^{\varphi(k)t}-1$ and the statement follows directly from item ($a$) of Lemma \ref{prop: criteria}.
	\end{proof}	
	
	The following is a direct consequence of the proposition.
	
	\begin{coro} {\label{coro: mas GP enteros}}
		Let $p$ be a prime. For any $\ell \in \N$ and any set of primes $r_1,\ldots,r_\ell$ with $r_i>p$ for $i=1,\ldots,\ell$ we have the infinite family of integral GP-graphs
		$$ \big\{ \G \big( r_1^{e_1} \cdots r_\ell^{e_\ell}, p^{t \{r_1^{e_1-1}(r_1-1) \cdots r_\ell^{e_\ell-1}(r_\ell-1) \} } \big)  \big\}_{t\in \N, (e_1, \ldots, e_\ell) \in \N^\ell} .$$
	\end{coro}
	
	\begin{proof}
		If $r_1,\ldots, r_n$ are primes bigger than $p$, then $k= r_1^{e_1} \cdots r_\ell^{e_\ell}$ is coprime with both $p$ and $p-1$ for every $(e_1, \ldots, e_\ell) \in \N^\ell$. 
		Thus, by Proposition \ref{prop: families of integral GPs} we have that $\G(k,p^{\varphi(k)t})$ is integral for every $t$. 
		The final result is obtained using that $\varphi$ is a multiplicative function and that $\varphi(r^e)=r^{e-1}(r-1)$ for every prime $r$.  
	\end{proof}

	\subsubsection*{Towers of integral GP-graphs}
	Now, we show that once we have an integral GP-graph defined over a finite field $\ff_q$, we automatically have an infinite sequence of integral GP-graphs defined over all the finite fields $\ff_{q^a}$ with $a\in \N$.
	
	\begin{prop} \label{prop: towers of integral GPs}
		Consider $\G(k,q)$ with $q=p^m$, $p$ prime, $k\mid q-1$ and $m\in\N$.
		If $Spec(\G(k,q))\subset \Z$, then 
		$Spec(\G(k\frac{q^{a}-1}{q-1},q^{a}))\subset \Z$ for any $a\in \mathbb{N}$.
	\end{prop}	
	
	\begin{proof}
		This follows directly from Theorem 2.1 in \cite{PV18}, since in this case we have that  $\G(k(\frac{q^{a}-1}{q-1}),q^{a})$ is a disjoint union of $q^{a-1}$-copies of the GP-graph $\G(k,q)$.
	\end{proof}
	
	Now, by using the previous result, we give an infinite family of integral GP-graphs.
	\begin{exam} \label{exam: integral families2}
		For any prime number $p$, 
		$$ \{ \G(\tfrac{p^{2t}-1}{p-1}, p^{2t})\}_{t\in \mathbb{N}}$$ 
		is an infinite family of integral GP-graphs, by the previous proposition, since item $(iii)$ of Example 5.4, the graph $\G(p+1,p^{2})$ is integral and $(p+1)\frac{p^{2t}-1}{p^{2}-1}=\frac{p^{2t}-1}{p-1}$.
		\hfill $\diamond$		
	\end{exam}	
	
	Notice that the previous proposition can be applied to any statement or example of this section to produce more general examples of integral GP-graphs.

	\subsection{Cyclotomic polynomials}
	Let $\mathcal{U}_n = \{w\in \C : w^n=1\}$ be the group of complex $n$-th roots of unity and let $\mathcal{U}_n^* = \{ \omega \in \mathcal{U}_n : \textrm{ord}(\omega)=n\}$ be the subgroup of primitive $n$-th roots of unity.
	The $n$-th \textit{cyclotomic polynomial} is defined by  
	$$\Phi_n(x) = \prod_{\omega \in \mathcal{U}_n^*} (x-\omega) = \prod_{ 0 < d< n, \, (d,n)=1 } \big(x-e^{\frac{2\pi i d}{n}} \big).$$ 
	
	Next, we give a criterion using cyclotomic polynomials for the construction of integral GP-graphs.
	
	\begin{lem} \label{lem: criteria2}
		Consider the GP-graph $\G(k,q)$ with $q=p^m$, $p$ prime, $k\mid q-1$ and $m\in\N$.
		If $k \mid \Phi_d(p)$ for some $d \mid m$ with $d>1$, then $Spec(\G(k,q))\subset \Z$.
	\end{lem}	
	
	\begin{proof}
		Since $\mathcal{U}_n = \bigcup_{d\mid n} \mathcal{U}_d^*$, the polynomial $x^n-1$ can be factored by cyclotomic polynomials of order $d$ dividing $n$, namely 
		\begin{equation} \label{eq: prod cycs}
			x^n-1 = \prod_{d\mid n} \Phi_d(x).
		\end{equation}	
		Since $\Phi_1(x)=x-1$ we have that 
		$$ \frac{p^m-1}{p-1} = \prod_{d\mid m, \, d>1} \Phi_d(p). $$
		The result thus follows directly by this and the hypothesis, since $\G(k,q)$ is integral if and only if $k$ divides $\frac{q-1}{p-1}$.
	\end{proof}

	\begin{exam}
		For instance, if we take $m=6$ then the $d$-th cyclotomic polynomials with $d\mid m$ and $d>1$ are
		$$\Phi_{2}(x)=x+1, \qquad \Phi_{3}(x)=x^{2}+x+1, \qquad \Phi_{6}(x)=x^{2}-x+1.$$
		Hence, for any prime $p$ we have the family of integral GP-graphs
		$$ \{\G(p+1,p^{6}),\, \G(p^{2}+p+1,p^{6}),\, \G(p^{2}-p+1,p^{6})\}. $$ 
		For the first four odd primes $p=3,5,7,11$, we have that
		\begin{alignat*}{3}
			&\G(4,3^6), &\qquad &\G(7,3^6),& \qquad &\G(13,3^6), \\
			&\G(6,5^6), &\qquad &\G(21,5^6),& \qquad &\G(31,5^6), \\
			&\G(8,7^6), &\qquad &\G(43,7^6), &\qquad &\G(57,7^6), \\
			&\G(13,11^6),&\qquad &\G(111,11^6),&\qquad &\G(133,11^6), 
		\end{alignat*}
		are integral GP-graphs.
		\hfill $\diamond$
	\end{exam}

	Actually, this allows us to produce infinite families of integral GP-graphs as follows.
	
	\begin{prop} \label{prop: cyclotomic}
		For any prime $p$ and any natural $d>1$,   
		$$ \{\G(\Phi_d(p), p^{dt}) \}_{t\in \N} $$ 
		is an infinite family of integral GP-graphs.
	\end{prop}
	
	\begin{proof}
		The first family follows immediately from Lemma \ref{lem: criteria2}. 
	\end{proof}

	By using the first cyclotomic polynomials we get the following infinite families of integral GP-graphs
	\begin{exam}
		For any prime $p$ and any $t\in \N$ we have the integral GP-graphs:
		\begin{align*}
			& \G(p-1,p^t), & & \G(p^4+p^3+p^2+p+1,p^{5t}),\\
			& \G(p+1,p^{2t}), 	& &\G(p^2-p+1,p^{6t}), \\
			& \G(p^2+p+1,p^{3t}), & & \G(p^6+p^5+p^4+p^3+p^2+p+1,p^{7t}), \\ 
			& \G(p^2+1,p^{4t}), & & \G(p^4+1,p^{8t}).
		\end{align*}
		Note that some of them were obtained previously in another non-systematic way.
		\hfill $\diamond$
	\end{exam}
	
	Using lists of cyclotomic polynomials or using their properties to compute them, one can give as many examples of integral GP-graphs as desired. For instance, for any $d \in \N$, $\Phi_d(x)$ can be recursively computed from \eqref{eq: prod cycs} and hence $\G(\Phi_d(p), p^{td})$ is integral for any prime $p$ and any $t\in \N$.
	
	We will only give some general families. 
	\begin{coro} \label{coro: mas ciclotomicos}
		Let $p,r$ be primes and $d, \ell \in\N$. The following are infinite families of integral GP-graphs:
		\begin{enumerate}[$(a)$]
			\item $\{ \G(1 + p+ p^2 + p^3 + \cdots + p^{r-1},p^{rt}) \}_{t\in \N}$ and $\{ \G(1-p+p^2-p^3+\cdots+p^{r-1},p^{2rt}) \}_{t\in \N}$. \msk 
			
			\item $\{ \G(p^{2^{d-1}}+1,p^{2^d t}) \}_{t\in \N}$. \msk  
			
			\item $\Big \{ \G \big( \sum\limits_{j=0}^{p-1} p^{jr^{d-1}}, p^{r^d t} \big) \Big \}_{t\in \N}$ and $\Big \{ \G \big( \sum\limits_{j=0}^{p-1} (-1)^j p^{j 2^{\ell-1} r^{d-1}}, p^{2^\ell r^d t} \big) \Big \}_{t\in \N}$ with $r$ odd. 
		\end{enumerate}
	\end{coro}

	\begin{proof}
		All expressions follow by Proposition \ref{prop: cyclotomic} and identities of cyclotomic polynomials. 
		
		Namely, item $(a)$ follows by $\Phi_d(x) = \frac{x^p-1}{x-1}=x^{p-1}+\cdots +p^2+p+1$ and 
		the fact that $\Phi_{2d}(x)=\Phi_d(-x)$.
		For item $(b)$ we use that $\Phi_{2^d}(x)=x^{2^{d-1}}+1$.
		Finally, the identities
		$$ \Phi_{r^d}(x) = \sum_{j=0}^{p-1} x^{jr^{d-1}} \qquad \text{and} \qquad \Phi_{2^\ell r^d}(x) = \sum_{j=0}^{p-1} (-1)^j x^{j 2^{\ell-1}r^{d-1}}$$
		imply item $(c)$.
	\end{proof}

	The Proposition \ref{prop: towers of integral GPs} can be applied to any integral GP-graph defined over a field $\ff_q$ to get infinite integral GP-graphs over all the fields $\ff_{q^a}$ for every $a\in \N$. As a summary of the results of the section, we have the following

	\begin{thm} \label{thm: general integral GP-graphs}
		Let $p$ be a prime and $k \in \N$. We have the general infinite families of integral GP-graphs:
		\begin{enumerate}[$(a)$]
			\item $\{\G(k \frac{q^{ta}-1}{q^t-1}, q^{at})\}_{a,t\in \mathbb{N}}$ where $q=p^k$, provided that $k\mid p-1$. \msk
			
			\item $\{\G(k \frac{q^{ta}-1}{q^t-1}, q^{at})\}_{a,t\in \mathbb{N}}$ where $q=p^2$, provided that $k\mid p+1$. \msk
			
			\item $\{\G(k \frac{q^{ta}-1}{q^t-1}, q^{at})\}_{a, t\in \mathbb{N}}$ 
			where $q=p^{\varphi(k)}$, provided that $\gcd(k,p(p-1))=1$. \msk 
			
			\item $\{\G(\Phi_d(p) \frac{q^{ta}-1}{q^t-1}, q^{at}) \}_{a,t\in \N}$ with $q=p^d$ and $d>1$. 
		\end{enumerate}
	\end{thm}

	\begin{proof}
		This follows by applying Proposition \ref{prop: towers of integral GPs} to Propositions \ref{prop: finite families of integral GPs}, \ref{prop: families of integral GPs} and \ref{prop: cyclotomic}.
	\end{proof}
	
	As we mentioned before, one can also apply Proposition \ref{prop: towers of integral GPs} to Corollaries \ref{coro: cadenitas}, \ref{coro: mas GP enteros} and \ref{coro: mas ciclotomicos} or to any example of the section to get the more general expressions. We prefer not to do that in the statements for clarity. 
	
	We finish the section with the following.
	\begin{rem}
		The only general known source of integral GP-graphs are semiprimitive GP-graphs (this includes strongly regular graphs) or Hamming GP-graphs. Except for some border cases, the integral GP-graphs obtained in this section with our methods (Euler totient function, cyclotomic polynomials) are neither semiprimitive nor Hamming GP-graphs and hence a new general source of integral GP-graphs.
	\end{rem}

	\section{GP-graphs with 3 eigenvalues} \label{sec:3}
	Here, we study GP-graphs with three different eigenvalues. We will characterize (conjecturally) all undirected GP-graphs having three eigenvalues and all directed GP-graphs with three eigenvalues.
	
	It is well-known that if $G=(V,E)$ is a regular graph with two different eigenvalues, then it is a complete graph 
	or a disjoint union of copies of a complete graph. This can be deduced from the fact that if a connected graph $G$ has $d+1$ different eigenvalues, then its diameter is at most $d$. 
	On the other hand, a strongly regular graph is connected with 3 different eigenvalues and, conversely, if a connected regular graph has 3 distinct eigenvalues then it is strongly regular (see \cite{BrV}). 
	So, a regular graph having exactly 3 different eigenvalues is the disjoint union of copies of a single strongly regular graph.

	Relative to GP-graphs we have the following. By the previous comments, those GP-graphs having 2 different eigenvalues must be a disjoint union of complete graphs, i.e. 		
	$$ \G(\tfrac{p^{m}-1}{p^{a}-1},p^{m}) = K_{p^{a}} \sqcup \cdots \sqcup K_{p^a} \qquad (\text{$p^{m-a}$ times}) $$ 
	with $p$ prime and $1\le a \le m$ (see  Theorem 2.1 in \cite{PV18}).
	
	What about GP-graphs with 3 different eigenvalues? Can we characterize those GP-graphs?
	As a warm-up, we present a family of GP-graphs (both directed and undirected) having exactly 3 different eigenvalues, which are disjoint unions of Paley graphs. 
	
	\begin{prop} \label{prop:paleys} 
		Let $p$ be an odd prime and $q=p^m$ with $m \in \N$ and let $a\mid m$.
		Let $\Gamma=\Gamma(\tfrac{2(p^m-1)}{p^a-1},p^m)$.
		Then, we have the connected components decomposition
		\begin{equation} \label{eq: Paleys} 
			\G \big( \tfrac{2(p^m-1)}{p^a-1},p^m \big) \simeq P_{p^a} \sqcup \cdots \sqcup P_{p^a} \qquad (\text{$p^{m-a}$ times}) 
		\end{equation}
		with $a=ord_n(p)$, where $n= \frac{p^a-1}2$, with spectrum given by
		\begin{equation} \label{eq: SpecG3}
			Spec(\G) = \begin{cases}
				\{ [\tfrac{p^a-1}2]^{p^{m-a}}, [\tfrac{-1+\sqrt{p^a}}2]^{\frac{p^a-1}2 p^{m-a}}, [\tfrac{-1-\sqrt{p^a}}2]^{\frac{p^a-1}2 p^{m-a}} \} & \: \text{if $p \equiv 1 \, (4)$}, \\[2mm]
				\{ [\tfrac{p^a-1}2]^{p^{m-a}}, [\tfrac{-1+i^a\sqrt{p^a}}2]^{\frac{p^a-1}2 p^{m-a}}, [\tfrac{-1-i^a\sqrt{p^a}}2]^{\frac{p^a-1}2 p^{m-a}} \} & \: \text{if $p \equiv 3 \, (4)$}.
			\end{cases}
		\end{equation}	
		Conversely, if the GP-graph $\G(k,p^{m})$, with $k\mid p^{m}-1$, is a disjoint union of Paley graphs $P_q$, 
		then $k=\tfrac{2(p^m-1)}{p^a-1}$, $q=p^m$ and $\G(k,p^{m})$ is as in \eqref{eq: Paleys}.
	\end{prop}
	
	\begin{proof}
		For the decomposition, one checks that $k_a=\frac{p^a-1}{n}=2$ and hence we have  
		$$\G(\tfrac{2(p^m-1)}{p^a-1},p^m) \simeq \G(2,p^a) \cup \cdots \cup \G(2,p^a)$$ 
		with $p^{m-a}$ components, by Theorem 2.1 from \cite{PV18}. Since  
		$\G(2,p^a) = P_{p^a}$ we get \eqref{eq: Paleys}. 
		
		Conversely, if $\G(k,p^{m})$ with $k\mid p^{m}-1$ is a union of classic Paley graphs $P_r$ we must have that $k= 2\tfrac{p^m-1}{p^a-1}$, $r=p^a$ and $\G(k,p^{m})$ is as in \eqref{eq: Paleys}. 
		
		Now, the assertion for $Spec(\G)$ follows from the decomposition \eqref{eq: Paleys} and the observation in ($ii$) of Remark 2.2 from \cite{PV18}.
	\end{proof}
	
	It is worth mentioning that, by \eqref{eq: SpecG3}, we have that
	$$ Spec \big( \G(\tfrac{2(p^m-1)}{p^a-1},p^m) \big) \subset 
	\begin{cases}
		\text{$\Z$} & \qquad \text{if $p\equiv 1 \!\!\! \pmod 4$ and $a$ even}, \\[1mm]
		\text{$\R$} & \qquad \text{if $p\equiv 1 \!\!\! \pmod 4$ and $a$ odd}, \\[1mm]
		\text{$\R$} & \qquad \text{if $p\equiv 3 \!\!\! \pmod 4$ and $a$ even}, \\[1mm]
		\text{$\C$} & \qquad \text{if $p\equiv 3 \!\!\!\pmod 4$ and $a$ odd}. \end{cases} $$
	That is, $\G$ has real spectrum when $\G$ is the union of undirected Paley graphs $\mathcal{P}_{p^a}$ and has complex spectrum when $\G$ is the union of directed Paley graphs $\vec{\mathcal{P}}_{p^a}$.

	\subsection{All undirected GP-graphs with 3 eigenvalues?}
	We know that semiprimitive GP-graphs are strongly regular graphs, hence undirected, connected, and with 3 different eigenvalues (see \cite{BrV}).
	Schmidt and White, using the generalized Riemann hypothesis, conjectured that all 2-weight irreducible cyclic codes $\mathcal{C}(k,q)$ come from one of three families: the semiprimitive codes, the subfield subcodes, and 11 exceptional codes (\cite{SW}).
	
	In \cite{PV2}, we showed that 2-weight irreducible cyclic codes are in spectral correspondence with the GP-graphs $\G(k,q)$; namely, weights and frequencies correspond with eigenvalues and multiplicities, respectively.
	The semiprimitive GP-graphs together with the 11 sporadic cases (coming from the 11 exceptional 2-weight irreducible cyclic codes, see \cite{PV2}) are all the strongly regular GP-graphs (see \cite{SW}, also \cite{PV18}).
	So, if the conjecture of Schmidt and White on 2-weight irreducible cyclic codes is true, then all undirected GP-graphs with 3 eigenvalues are the semiprimitive and the sporadic ones. This leads us to study directed GP-graphs with 3 eigenvalues.

	\subsection{All directed GP-graphs with 3 eigenvalues}
	Now, we characterize all directed GP-graphs with exactly three different eigenvalues (i.e., with two different complex non-real eigenvalues besides the regularity degree $n=\frac{q-1}k$). It turns out that the only such GP-digraphs are the ones obtained in Proposition \ref{prop:paleys} in the directed case.

	We will use the following notation. Given a graph $\G$, let $\mu_\G$ denote the \textit{number of different eigenvalues} of $\G$, that is 
	$$\mu_\G = \# (Spec(\G))$$ 
	thinking the spectrum as a set instead of a multiset. In the notation of \eqref{eq: spec} we have
	\begin{equation} \label{eq: mu G}
		\mu_\G = t+1,
	\end{equation}
	where $t$ is the number of non-principal different eigenvalues.
	
	As a consequence of Theorems \ref{thm: real spec} and \ref{thm: nature spec v2}, we obtain the characterization of all 
	directed GP-graphs having exactly 3 different eigenvalues. Namely, the only such graphs are those that are disjoint unions of directed Paley graphs $\vec{\mathcal{P}}_q$.

	\begin{thm} \label{thm: directed SRG}
		Let $q=p^m$ be an odd prime power with $m\in \N$ and let $k \in \N$ such that $k\mid q-1$.
		If $\G=\G(k,q)$ is directed, then $\mu_\G \ge 3$. 
		Moreover, $\mu_\G = 3$ if and only if $k=2\frac{p^m-1}{p^{a}-1}$ where 
		$a=ord_{n}(p)$ with $n=\frac{q-1}{k}$ and $p^a \equiv 3 \pmod 4$. 
		In this case we have that $a$ is odd, $p \equiv 3 \pmod 4$, and 
		\begin{equation} \label{eq: union P7s}
			\G = \G(2\tfrac{p^m-1}{p^{a}-1}, q) = \vec{\mathcal{P}}_{p^a} \sqcup \cdots \sqcup \vec{\mathcal{P}}_{p^a} \qquad 
			(\text{$p^{m-a}$ times})	
		\end{equation} 
		with complex spectrum given by 
		$$	Spec(\G) = \begin{cases}
			\{ [\tfrac{p^a-1}2]^{p^{m-a}}, [\tfrac{-1 + i p^{\frac{a-1}2} \sqrt{p}}2]^{\frac{p^a-1}2 p^{m-a}}, [\tfrac{-1 - i p^{\frac{a-1}2} \sqrt{p}}2]^{\frac{p^a-1}2 p^{m-a}}  \} & \: \text{if $a \equiv 1 \, (4)$}, \\[2mm]
			\{ [\tfrac{p^a-1}2]^{p^{m-a}}, [\tfrac{-1 - i p^{\frac{a+1}2} \sqrt{p}}2]^{\frac{p^a-1}2 p^{m-a}}, [\tfrac{-1 + i p^{\frac{a+1}2} \sqrt{p}}2]^{\frac{p^a-1}2 p^{m-a}} \} & \: \text{if $a \equiv 3 \, (4)$}.
		\end{cases}
		$$ 
	\end{thm}

	\begin{proof}
		We will consider the cases when $\G$ is connected and disconnected separately. 
		
		Assume first that $\G$ is connected.
		Since $\G$ is directed, by Theorem \ref{thm: nature spec v2} we have that $v_{2}(k)= v_{2}(q-1)$ and hence 
		that $\G(\frac{k}{2},q)$ is undirected. 
		Moreover, the connection sets of these graphs satisfy \eqref{eq: Rk Rk2}, that is
		$R_{\frac k2} = R_k \sqcup (-R_k)$. The relation of the spectra of $\G(\frac k2,q)$ and $\G(k,q)$ is given in Theorem~\ref{thm: XGSy-S}. In fact, by \eqref{eq: chi(S)} and \eqref{eq: chi2Re}, the eigenvalues of $\G(\kappa,q)$ are given by 
		$$ \lambda_{\chi} = \chi(S) = \sum_{g \in S} \chi(g) $$ 
		with $\kappa = \frac k2,k$ and $S=R_{\frac k2}, R_k$, respectively, where $\chi$ is an additive character of $\ff_q$ and thus we have 
		$$ \chi(R_{\frac k2}) =  2\, \mathrm{Re}(\chi(R_{k})) .$$
		
		Hence, we have the following map between the spectra of the graphs $\G(k,q)$ and $\G(\tfrac{k}{2},q)$:
		$$ \Phi: Spec(\G(k,q)) \rightarrow Spec(\G(\tfrac{k}{2},q)), \qquad \Phi(\lambda)=2 \mathrm{Re}(\lambda).$$ 
		By Theorem \ref{thm: real spec}, there exists an eigenvalue of $\G(k,q)$ which is non-real. Let $\chi$ be the corresponding additive character of $\mathbb{F}_{q}$ such that $\lambda_\chi=\chi(R_{k})\not \in \mathbb{R}$, notice that if $\chi^{-1}$ is the inverse additive character of $\chi$, then we have that 
		$$\Phi(\lambda_\chi) = \Phi(\lambda_{\chi^{-1}})$$ 
		with $\lambda_{\chi}\neq \lambda_{\chi^{-1}}$. 
		Thus, we obtain that
		\begin{equation}\label{eq: mu nd dir}
			\mu_{\G(k,q)}\ge \mu_{\G(\frac{k}{2},q)}+1.
		\end{equation}
		Since $\mu_{\G(\frac{k}{2},q)}\ge 2$, the equation \eqref{eq: mu nd dir} implies that $\mu_{\G(k,q)}\ge 3$, proving the first assertion.
		
		Now, assume that $\mu_\G=2$ and suppose by contradiction that $k\ne 2$. 
		Then, $\frac{k}{2}\neq 1$ and so the graph $\G(\frac{k}{2},q)$ is not the complete graph. Hence, $\mu_{\G(\frac{k}{2},q)}\ge 3$, since the complete graph is the unique connected graph with two different eigenvalues, this implies that
		$\mu_{\G(k,q)}\ge 4$.
		
		Conversely, if $k=2$ with $q\equiv 3 \pmod{4}$ then the graph $\G(2,q)$ is $\vec{\mathcal{P}(q)}$, which has three different eigenvalues (see ($b$) in Example \ref{exam: paley}).
		
		\
		
		Now assume the general case, that is $\G$ is not necessarily connected. 
		By Theorem 2.1 from \cite{PV18}, we have that
		$$ \G(k,q)= \bigcup_{i} \G(k_a,p^{a}). $$ 
		with $k_{a}=\frac{p^{a}-1}{n}$, where $a=ord_{n}(p)$. This clearly implies that 
		$$ \mu_{\G(k,q)}= \mu_{\G(k_a,p^{a})}. $$ 
		The result follows directly from Propositions \ref{prop:paleys} and Lemma \ref{thm: directed SRG}.
	\end{proof}

	This means that if $\G(k,q)$ is directed and $k > 2$, then it has four or more different eigenvalues.

	\begin{rem}
		We can use the above GP-graphs to obtain an example with more than one non-trivial real eigenvalue. 
		Indeed, let $p=7$, $m=3$ and take
		$$	k= \tfrac{2(7^{3}-1)}{3(7-1)}=38 .$$
		Then, the GP-graph $\G(38,343)$ is the Cartesian product of three copies of the directed Paley graph $ \G(2,7)=\vec{\mathcal{P}}_{7}$ (see Proposition 2.3 and Lemma 3.3 from \cite{PV7}), 
		that is 
		$$ \G(38,343) \simeq \vec{\mathcal{P}}_{7} \, \Box  \,   	\vec{\mathcal{P}}_{7} \,\Box \, \vec{\mathcal{P}}_{7}.$$ 
		In particular, its eigenvalues are all the $3$-sums of the eigenvalues of $\vec{\mathcal{P}}_{7}$. By Theorem \ref{thm: directed SRG} with $p=7$ and $a=m=1$, we have 
		$$ Spec(\vec{\mathcal{P}}_{7}) =
		\{ [3]^{1}, [\tfrac{-1 + i \sqrt{7}}2]^{3}, [\tfrac{-1 - i \sqrt{7}}2]^{3} \} .$$
		Thus, we obtain that $Spec(\G(38,343))$ is the multiset
		\begin{multline*}
			\{ [9]^{1}, [2]^{54}, [2 + i \sqrt{7}]^{27}, [2 - i \sqrt{7}]^{27}, [\tfrac{11 + i \sqrt{7}}2]^{9}, [\tfrac{11 - i \sqrt{7}}2]^{9}, \\  [\tfrac{-3 +3 i \sqrt{7}}2]^{27}, [\tfrac{-3 - 3i \sqrt{7}}2]^{27},[\tfrac{-3 + i \sqrt{7}}2]^{81}, [\tfrac{-3 - i \sqrt{7}}2]^{81} \} .
		\end{multline*} 
		
		Similarly, the GP-graph $\G(19,343)$ is the Cartesian product of three copies of the complete graph $\G(1,7)=K_7$, with $Spec(K_7) = \{[6]^1, [-1]^6\}$, and so
		$$ Spec(\G(19,343))= 	\{ [18]^{1}, [11]^{18}, [4]^{108}, [-3]^{216} \}. $$
		
		Now, notice that if we take the function $\Phi(\lambda)=2 \mathrm{Re}(\lambda)$ from $Spec(\G(38,343))$ to $Spec(\G(19,343))$ we have that the fibers (without taking into account the multiplicities) satisfy
		$$ |\Phi^{-1}\{18\}|=1, \qquad |\Phi^{-1}\{11\}|=2, \qquad 	|\Phi^{-1}\{4\}|=3, \qquad 	|\Phi^{-1}\{18\}|=4. $$
		In this case $\mu_{\G(19,343)}=4$ and $\mu_{\G(38,343)}=10$. This leads  to the following \newline

		\noindent \textbf{Question:} \textit{Are there good (upper-lower) bounds for 
			$\mu_{\G(2k,q)}-\mu_{\G(k,q)}$ 
			when $k\mid \frac{q-1}{2}$ and $k\nmid \frac{q-1}{4}$?}
	\end{rem}

	We now make some comments on the relation with strongly regular graphs.
	
	\begin{rem}
		($i$) A \textit{strongly regular graph}  (SRG for short) with parameters $(n,k,e,d)$ is a graph such that its adjacency matrix satisfies
		$$
		A^{2}=kI+eA+d(J-I-A),\qquad AJ=JA=kJ,
		$$
		where $I$ is the identity matrix and $J$ is the all $1$'s matrix. Equivalently, if $N_v$ denotes the set of neighbors of $v$, then a graph is strongly regular if it a connected regular graph such that for any pair of vertices $v,w$, the size $|N_v \cap N_w|$ is equal to $e$ or $d$ depending on if $v$ and $w$ are neighbors or not, respectively. In terms of eigenvalues, it is well known that any strongly regular graph has $3$ different eigenvalues and conversely any connected graph with three different eigenvalues is strongly regular. 
		
		Thus, there are many strongly regular GP-graphs. For instance, we have seen in the previous section that $\G(1,q)=K_q$, $\G(2,q)=\mathcal{P}_q$, $\G(3,q)$ and $\G(4,q)$ in certain cases ($q=p^m$ with $p\equiv -1 \pmod k$ for $k\in \{3,4\}$ and $m$ even), and $\G(\frac{q+1}2,q^2)$ are SRGs. In fact, all these examples are particular cases of semiprimitive GP-graphs, which are always strongly regular graphs. 
		
		($ii$) In the directed case, there exists a notion of strongly regularity (see \cite{Du}, \cite{KMMZ}): a \textit{directed strongly regular graph} (DSRG for short) with parameters $(n,k,t,e,d)$ is a digraph such that 
		$$
		A^{2}=tI+eA+d(J-I-A), \qquad AJ=JA=kJ.
		$$  
		In this case, each vertex of the digraph still has in- and out-degree $k$, but now with only
		$t$ edges being undirected, leaving $k-t$ edges coming in only and $k-t$ coming out only. 
		The interpretations of $e$ and $d$ remain the same as in undirected strongly regular graphs. 
		
		This definition has not have the same interpretation on the number of its different eigenvalues as in the undirected case. Moreover, it can be seen that there are no Cayley graphs over abelian groups which are DSRG (see \cite{BrH}). Thus, in this case, the directed GP-graphs with $3$ eigenvalues that we found in the above section cannot be DSRG. 
	\end{rem}

	\section{Periods of GP-digraphs}	
	\label{sec: GP-digraphs}
	
	In this section, we focus on the study of the period of directed GP-graphs, and their relation with the number of eigenvalues with maximum absolute value.

	\subsection{Almost all directed GP-graphs have period 1} \label{subsec: period}
	Suppose that $G$ is a directed graph (digraph) and let $d$ denotes the greatest common divisor between all the lengths $\ell(\vec{C})$ of the directed cycles $\vec{C}$ in $G$. The integer $d=d(G)$ is called the \textit{period} of $G$. In symbols,
	\begin{equation} \label{index imprim}
		d = d(G)= \gcd \{ \ell(\vec{C}) : \text{ $\vec{C}$ is a directed cycle in $G$}\}.
	\end{equation}
	Clearly $d(\vec{C}_\ell)=\ell$.
	It is customary to say that the graph is \textit{primitive} if $d=1$, and we will also say that it is \textit{semiprimitive} if $d=2$.

	\begin{rem}
		We can generalize the definition of period (and hence that of primitivity) to undirected graphs in at least two forms. Suppose $G$ is undirected. One way to define the period of $G$ is as the greatest common divisor between all the lengths of the cycles in $G$. Another way is to consider any undirected edge as a pair of directed edges (arrows) with different orientations
		and use the definition of period already given for directed graphs. 
		Notice that, with this last definition, any undirected graph $G$ is \textit{primitive} if it has some odd-length cycle or \textit{semiprimitive} otherwise. That is, if $\G$ is an undirected graph then
		$$d(\G)= \begin{cases}
			1, & \qquad \text{if $\G$ is non-bipartite,} \\[1mm]
			2, & \qquad \text{if $\G$ is bipartite.}
		\end{cases}$$ 
		
	\end{rem}

	The following proposition asserts that the only non-primitive, strongly connected, directed GP-graph is $\G(p-1,p)$ with $p$ an odd prime  (recall that the GP-graphs $\G(k,2^m)$ are undirected).
	\begin{prop} \label{prop: realspec con}
		Let $q=p^m$ with $p$ an odd prime and $m\in \N$ and let $k \in \N$ be such that $k\mid q-1$. 
		Suppose that $\G(k,q)$ is directed and strongly connected.
		Then $\G(k,q)$ has period $1$ with the only exception of $\G(p-1,p)$ which has period $p$. 
	\end{prop}

	\begin{proof}
		Recall that $\G(k,q)$ is strongly connected if and only if $\frac{q-1}{k}\dagger q-1$ (i.e., $a=1$ in Theorem 2.1 from \cite{PV18}) and that $\G(k,q)$ is directed if and only if $k\nmid \frac{q-1}{2}$ with $q$ odd.
		
		Suppose that $q=p$ is odd and $k=p-1$. Then, by Corollary 3.1 from \cite{PV18} we have that 
		$$\G(p-1,p)=\vec{C_p}$$ 
		has only one directed cycle of length $p$ and hence has period $p$.
		
		To prove the rest of the assertion, that is that $\G(k,q)\ne \G(p-1,p)$ has period $1$, we split the proof into two cases: $q$ is not prime or $q=p$ but $k \ne p-1$. Recall from \eqref{eq: Gkq} that $\G(k,q)=Cay(\ff_q,R_k)$ where $R_k =\{ x^{k} : x \in \ff_{q}^*\}$.
		
		\subsection*{Case $1$: $q=p^m$ with $m>1$.}
		Since $\G(k,q)$ is directed then $k\nmid \frac{q-1}{2}$. 
		Notice that $|R_k|=1$ if and only if $q=p$ and $k=p-1$ since $\frac{q-1}{k}\dagger q-1$.
		On the other hand, $k\nmid \frac{q-1}{2}$ if and only if $2\nmid \frac{q-1}{k}$. 
		Hence, in this case we have that $|R_k|=\frac{q-1}{k}\ge 3$.
		
		\sk	
		\noindent
		\textit{Claim:} The digraph $\G(k,q)$ contains a directed cycle of length an odd prime $r$ with $r \mid \frac{q-1}{k}$.

		The elements of $R_k$ are exactly the roots of the polynomial $p(x)=x^{\frac{q-1}{k}}-1\in \ff_{q}[x]$, that is
		$p(x)=\prod_{\omega\in R_k}(x-\omega)$. 
		In particular, the sum of all the elements in $R_k$ coincides with the second leading coefficient 
		of $p(x)$ and so
		\begin{equation*} \label{eq: Rk}
			\sum_{\omega\in R_k}\omega=0.
		\end{equation*}
		
		Now, notice that if $r \mid \frac{q-1}{k}$, i.e.\@ $\frac{q-1}{k} = rt$ for some $t \in \mathbb{N}$, 
		then $R_{kt} \subseteq R_k$. 
		As before, the elements of $R_{kt}$ are exactly the roots of the polynomial $q(x)=x^{\frac{q-1}{kt}}-1\in \ff_{q}[x]$, 
		that is
		$ q(x) = \prod_{\omega\in R_{kt}}(x-\omega)$,  
		and hence there exists $r=|R_{kt}|$ elements in $R_k$ such that 
		\begin{equation} \label{eq: Rkt}
			\sum_{\omega\in R_{kt}}\omega=0.
		\end{equation}

		Since $k\nmid \frac{q-1}{2}$, then $2\nmid \frac{q-1}{k}$ and so $\frac{q-1}{k}$ has only odd factors.
		Then, there exists an odd prime $r$ such that $r \mid \frac{q-1}{k}$. 
		In this case, 
		$$R_{\frac{q-1}{r}} = \langle \alpha \rangle,$$ 
		where $\alpha = \omega^{\frac{q-1}{r}}$ is an element of order $r$ in $\ff_{q}$, 
		since $\omega$ is a primitive element of $\ff_{q}$. 
		
		Let us see that $0, 1, 1+\alpha, 1+\alpha+\alpha^{2}, \ldots, 1+\alpha+\cdots+\alpha^{r-2}$ and $1+\alpha+\cdots+\alpha^{r-1}$ form a directed cycle in $\G(k,q)$. 
		Indeed, since $\alpha\in R_{k}$, any power $\alpha^i$ is also in $R_k$. So, the powers of $\alpha$ induce arrows in the graph $\G(k,q)$; namely, there is an arrow between any vertex $x$ and $y=x+\alpha^i$ for any $i$.
		Thus, we have the following directed walk in $\G(k,q)$.
		\begin{equation} \label{eq: cycle}
			0, \: 1, \: 1+\alpha, \: 1+\alpha+\alpha^{2}, \: \ldots, \: 1+\alpha+\cdots+\alpha^{r-2}, \: 1+\alpha+\cdots+\alpha^{r-1}
		\end{equation}	
		and, since $1+\alpha+\cdots+\alpha^{r-1}=0$ by \eqref{eq: Rkt}, it is a closed walk in $\G(k,q)$ of length $r$. 
		
		It is enough to see that all of the intermediate vertices $1+\alpha+\cdots +\alpha^i$ with $i\ne r-1$ are all different.
		Suppose, by contradiction, that there are $i,j \in \{1,\ldots,r-1\}$ with $i<j$ such that
		$$ 1+\alpha+\cdots +\alpha^i = 1+\alpha+\cdots +\alpha^i + \alpha^{i+1} + \cdots +\alpha^j. $$
		Then, we have that $\alpha^{i+1} + \cdots +\alpha^j=0$, that is $\alpha^{i+1} (1+\alpha + \cdots +\alpha^{j-i-1})=0$, 
		which implies that
		$${\sum_{\ell=0}^{j-i-1} \alpha^{\ell}=0}.$$ 
		Thus, $\alpha$ is a zero of the polynomial $s(x)=x^{j-i}-1$ and hence $\alpha^{j-i}=1$ with $j-i<r$, which is absurd. 
		In this way, we see that all of the intermediate vertices in \eqref{eq: cycle} are all different, and hence we obtain a directed cycle of length $r$, with $r$ an odd prime, as claimed. \hfill {\tiny $\blacksquare$}
		
		By the claim, there exists a directed cycle of an odd prime length $r$  with $r\mid \frac{q-1}{k}$ and so $r\mid q-1$. 
		Now, by taking into account that $p$ is prime and $\gcd(p,q-1)=1$ then $\gcd(p,r)=1$. 
		Since $\G(k,q)$ has a directed cycle of length $p$ (adding $p$ times $1$ to the vertex $0$),
		therefore the period of $\G(k,q)$ is $d=1$. \hfill $\diamond$
		
		\subsection*{Case $2$: $q=p$ odd and $k<p-1$.}
		If $p=3$, the only $k$ satisfying $k \nmid \frac{p-1}2$ is $k=2$ and hence $\G(2,3)=\vec{C}_3$, which has period 3. Hence, we assume that $p>3$.
		In this case, notice that we have a directed cycle of length $p$ obtained by successively adding $1$ to the vertex $0$, that is
		$$ 0, \: 1,\:  2=1+1,\:  \ldots,\, p-1 = \underbrace{1+\cdots+1}_{p-1 \text{ times}}, \: p=\underbrace{1+\cdots+1}_{p\text{ times}} = 0$$ 
		(since we are in characteristic $p$).

		On the other hand, since $|R_k|>1$, there exists an element $x^k \in R_k$ with $x^k \ne 1$ such that $x^k \equiv t \pmod{p}$ and $1 < t < p-1$ (since $p-1$ is not in $R_k$ due to the non-symmetry of $R_k$). 
		Hence, we have the following directed cycle 
		$$0, \: t, \: t+1, \: t+2, \: \ldots, \: t+(p-t-1), \: t+(p-t)=0$$
		of length $p-t+1<p$.
		Hence, the period of $\G(k,p)$ is $d=\gcd(p-t+1,p)=1$. \hfill $\diamond$
		
		Therefore, $\G(k,q)$ has period $d=1$, with the only exception of $\G(p-1,p)$ which has period $d=p$, as asserted.
	\end{proof}

	We now illustrate the Case 2 in the above proof, showing that there can be cycles of length less than $p$.
	
	\begin{exam}
		The smallest directed Paley graph is $\G(2,7)=P(7) = Cay(\ff_7,(\ff_7^*)^2))$. 
		Since the non-zero squares in $\ff_7$ are $1, 2$ and $4$ we see that $P(7)$ has directed $3$-cycles, $4$-cycles, $6$-cycles and $7$-cycles.
		For instance, we have the three directed $7$-cycles 
		$$(0,1,2,3,4,5,6,0), \quad (0,2,4,6,1,3,5,0), \quad (0,4,1,5,2,6,3,0)$$ 
		obtained by respectively adding seven $1$'s, seven $2$'s and seven $4$'s to $0$. Also, we have the directed $6$-cycle $(0,2,3,4,5,6,0)$ obtained by first adding $2$ to $0$ and then adding five $1$'s, the directed $4$-cycle $(0,4,5,6,0)$ obtained by adding $4$ to $0$ and then three $1$'s, and the directed $3$-cycle $(0,4,6,0)$ obtained by adding $4$ to $0$, then $2$ and then $1$, all starting from $0$. The remaining directed cycles are obtained similarly by permuting the additions or starting from other vertices.
		Hence, the period is $d=\gcd\{3,4,6,7\}=1$. 
		\hfill $\diamond$
	\end{exam}

	\begin{rem} \label{rem: dir/undir K2s}
		Although $\G(k,2^m)$ is undirected, if one considers $\G(1,2) \simeq K_2$ as $\overset{\rightarrow}{K}_2 \cup  \overset{\leftarrow}{K}_2$, then $\G(1,2)$ has period $2$, and hence $\G(2^m-1,2^m)$ has period 2 for any $m \in \N$.
	\end{rem}

	As a direct consequence of Theorem 2.1 from \cite{PV18} and Proposition \ref{prop: realspec con} we 
	now show that every GP-graph has a trivial period with the only exception of the graphs of the form $\G(q-1,q)$ with $q$ odd, having period $p$.

	\begin{thm} \label{thm: GPperiods}
		Let $q=p^m$ with $p$ an odd prime and $m\in \N$ and let $k \in \N$ such that $k\mid q-1$. 
		If $\G(k,q)$ is directed then it has period $1$, with the only exception of $\G(q-1,q)$, which has period $p$.
	\end{thm}

	\begin{proof}
		We know from Proposition 3.2 from \cite{PV18} that 
		$$ \G(q-1,q) = \vec{C}_p \cup \cdots \cup \vec{C}_p$$ 
		with $\vec{C}_p$ repeated $p^{m-1}$ times and hence
		$\G(q-1,q)$ has period $p$, as asserted.
		
		On the other hand, if $\frac{q-1}{k}\dagger q-1$ (i.e., if $\G(k,q)$ is strongly connected), the assertion is exactly Proposition \ref{prop: realspec con}.
		So, we can assume that $\frac{q-1}{k}$ is not a primitive divisor of $q-1$ (i.e., $\G(k,q)$ is not strongly connected). 
		Let $a$ be the minimal positive integer such that $n=\frac{q-1}{k}\mid p^{a}-1$.
		In this case, there exists $k_a =\frac{p^{a}-1}n$, 
		by Theorem 2.1 from \cite{PV18} we have that
		$$\G(k,q) \simeq \G(k_a,p^a) \cup \cdots \cup \G(k_a,p^a) \qquad \text{($p^{m-a}$ times)}.$$
		Since $\G(k_a,p^a)$ is strongly connected, Proposition \ref{prop: realspec con} 
		says that if $\G(k_a,p^a)$ has period $d>1$, then $a=1$ and $k_a=p-1$ with $p$ odd.
		Therefore, if $\G(k,q)$ has period $d>1$, then  $k=q-1$, as asserted.
	\end{proof}

	Summarizing the results of the section, 
	we have that the period of a directed GP-graph is given by
	\begin{equation} \label{eq: periods}
		d(\G(k,q)) = \begin{cases}
			1,  & \qquad \text{ if $k \ne q-1$}, \\[1mm]
			p,  & \qquad \text{ if $k  =  q-1$}.
		\end{cases}
	\end{equation}
	for $q=p^m$, with $p$ an odd prime.

	\subsection{Relation between periods and spectrum} \label{sec:5}
	Assume that $G$ is a directed graph with period $d$.
	It is known that $Spec(G)$, as a set of complex points, is invariant under a rotation about the origin by the angle $\frac{2\pi}d$ (see Theorem 2.1 in \cite{Br}).
	In particular, we deduce that if $G$ has real spectrum then it must has period $1$ or $2$, that is
	\begin{equation} \label{eq: d12}
		Spec(G) \subset \R \qquad \Rightarrow \qquad d(G) \in \{1,2\}.
	\end{equation}	
	
	Also, notice that $d(G)$ is even if and only if $G$ is bipartite, since in this case $G$ can only have directed cycles of even length 
	(see condition (B3') in Section 4 from \cite{PV18}, also Theorem 3.1 in \cite{Br}). 
	Hence, if $G$ is non-bipartite with real spectrum then it has period $1$. 
	Thus, in the case that interests to us, if $\G$ is a directed GP-graph, by Theorem 4.2 from \cite{PV18} and Remark \ref{rem: dir/undir K2s} we have that
	\begin{equation} \label{SpecR sii d=12}
		Spec(\G) \subset \mathbb{R} \quad \text{and} \quad \G \ne \G(2^m-1,2^m) \qquad \Rightarrow \qquad d(\G)=1,
	\end{equation}
	for $m \in \N$.
	
	It is a well-known fact that if $G$ is a $k$-regular graph, then $|\lambda|\le k$ for every eigenvalue $\lambda$ of $G$. 
	For GP-graphs, since $\frac{q-1}k$ is the regularity degree $\G(k,q)$, we have that $|\lambda| \le \frac{q-1}k$ for every eigenvalue $\lambda$ of $\G(k,q)$, that is 
	$$ Spec(\G(k,q)) \subset \bar{B}(0,\tfrac{q-1}k).$$
	
	In matrix theory, the \textit{index of imprimitivity} of an irreducible square matrix $A$, denoted $\delta(A)$, is the number of eigenvalues with absolute value equal to the spectral radius $\rho(A)$ of $A$. It is a well-known fact that the adjacency matrix $A_G$ of a directed graph $G$ is irreducible if and only if $G$ is strongly connected. 
	Moreover, in this case the period of $D$ coincides with the index of imprimitivity of $A_G$ (see Lemma 3.4.1 in \cite{BrRy},  probably first obtained here \cite{Ro}), 
	that is 
	\begin{equation} \label{eq: d=delta}
		d(G) = \delta(A_G).
	\end{equation} 
	In other words, an $n$-regular (strongly) connected directed graph $G$ has period $1$ if and only if 
	\begin{equation}\label{eq spec digraph}
		|\lambda|< n \qquad \text{for all $\lambda \in Spec (G) \smallsetminus\{n\}$.}
	\end{equation} 
	
	As a direct consequence of Theorem \ref{thm: GPperiods}, we obtain that the only GP-graph having complex non-real eigenvalues on the boundary of the closed ball, i.e.\@ on the circle $\tfrac{q-1}k \mathbb{S}^1$, is $\G(p-1,p)$ with $p$ an odd prime. 
	
	\begin{prop}  \label{prop: real spec S1}
		Let $q$ be a prime power and let $k\mid q-1$. Then, we have that 
		$$ Spec(\G(k,q)) \cap \tfrac{q-1}{k} \mathbb{S}^1  = \{\tfrac{q-1}{k}\} $$  
		except for $k=q-1$. 
	\end{prop}
	
	\begin{proof}
		Notice that if $k=q-1$ with $q=p^{m}$, by Theorem \ref{thm: GPperiods} and Theorem 4.2 from \cite{PV18} we have that 
		$$Spec(\G(q-1,q)) = \{ [1]^{p^{m-1}},[\zeta_p]^{p^{m-1}}, [\zeta_p^{2}]^{p^{m-1}}, \ldots, [\zeta_p^{p-1}]^{p^{m-1}} \},$$ 
		where $\zeta_p=e^{\frac{2\pi i}p}$. Therefore, the spectrum of $\G(q-1,q)$ intersects the punctured circle 
		$\mathbb{S}^1 \smallsetminus \{1\}$.
		
		Now assume that $k<q-1$ and let $\G=\G(k,q)$. 
		By Theorem 2.1 from \cite{PV18} it is enough to assume that $\G$ is connected in the wide sense (i.e., connected if $\G$ is undirected and strongly connected if $\G$ is directed), and hence we assume that this is the case.
		
		Assume first that $\G$ is undirected. 
		Hence $\G$ has real spectrum, and the only possible real eigenvalues on the circle $\tfrac{q-1}{k} \mathbb{S}^1$ are $\pm \tfrac{q-1}{k}$.
		We have that  
		$$ Spec(\G) \cap \big(\tfrac{q-1}{k} \mathbb{S}^1 \smallsetminus \{\tfrac{q-1}{k}\} \big) = \varnothing $$ 
		if and only if $\G$ is non-bipartite. By Proposition 4.1 from \cite{PV18}, this always happens except for the case $k=1$ and $p=2$.
		
		Now, assume that $\G$ is directed.
		By the definition of the index of imprimitivity for irreducible square matrices,
		if $\G$ is a strongly connected digraph, we have that 
		$$ \# \big( Spec (\G) \cap (\tfrac{q-1}{k} \mathbb{S}^1 \smallsetminus \{\tfrac{q-1}{k}\}) \big) = \delta(A_\G)-1 = d(\G)-1, $$
		where $\delta(A_\G)$ is the index of imprimitivity and $d(\G)$ is the period of $\G$ and we have used \eqref{eq: d=delta}. 
		In this case, 
		$ Spec (\G) \cap \big(\tfrac{q-1}{k} \mathbb{S}^1 \smallsetminus \{\tfrac{q-1}{k}\} \big) = \varnothing $ 
		if and only if $\G$ has period $1$.
		By Proposition \ref{prop: realspec con}, this always happens except for the case $k=p-1$ with $p$ an odd prime.  
		
		Thus, we have showed that if $\G$ is connected its spectrum only intersects the punctured circle $\tfrac{q-1}{k} \mathbb{S}^1 \smallsetminus \{\tfrac{q-1}{k}\}$ in the case $\G=\G(p-1,p)$ with $p$ prime. In the general case ($\G$ not connected), we obtain the same result for $\G=\G(q-1,q)$.
	\end{proof}

	\section{Application: Weak Waring numbers} \label{sec: Waring}
	As a quite unexpected application of our previous results, we now study weak Waring numbers over finite fields.

	We recall that, as a natural generalization of the classical Hilbert-Waring problem in $\N$, given $k\in \N$, the Waring number $g(k,q)$ over the finite field $\ff_q$ is the minimum $s\in \mathbb{N}$ (if exists) such that for any element $a \in \mathbb{F}_q$ there exist $x_1, \ldots, x_s \in \mathbb{F}_q$ with
	$$ a= x_1^k + \cdots + x_s^k. $$ 
	These numbers were studied by several people, see for instance the works of Cipra, Cochrane and Pinner, Moreno and Castro, Winterhoff et al, etc. For a complete list of results and references up to 2013 see Section 7.3.4 in Mullen-Panario's handbook of finite fields \cite{MP}.

	In a similar way, we have the following.
	\begin{defi} \label{def: wkq}
		Given $k\in \N$, we define the \textit{weak Waring number over finite fields} $w(k,q)$ 
		as the minimum $s\in \mathbb{N}$ (if exists) such that for any $a\in\mathbb{F}_{q}$ there exist 
		$x_1, \ldots, x_s \in \mathbb{F}_q$  
		such that
		$$ a= \pm x_1^k \pm \cdots \pm x_s^k,$$
		meaning that each term can have a plus or a minus sign independently. 
	\end{defi}

	It is clear from the definitions that if both numbers $w(k,q)$ and $g(k,q)$ exist, then 
	$$ w(k,q) \le g(k,q).$$
	Notice that, for instance, that $w(2,3)=1$  (since $2=-1$ in $\Z_3$) but $g(2,3)=2$, so the above inequality could be strict.

	\begin{rem}
		The weak Waring number was previously defined in the context of integers almost a century ago by Wright \cite{Wr}. He referred to it as an easier Waring problem. However, as the first paragraph of Borwein's chapter 12 of \cite{Bor} says, ``So to date, the easier Waring problem is not easier than Waring problem''.  
		Cochrane studied this number over $\Z_p$ first with Pinner in \cite{CP} and later over $\ff_q$ with Cipra in \cite{CiCo}. 
		They called it \textit{plus-minus Waring number} and denoted it by $\delta(k,p)$. They use the circle and the lattice method to study these numbers.
		
		In item ($e$) of Theorem 9 of \cite{CiCo} the authors obtained the following result (in our notations) relating the weak Waring number $w(k,q)$ with the Waring number $g(k,q)$. Namely, to get the neat relation:
		\begin{equation} \label{eq: Cipra wkq=gkq}
			w(k,q) = \begin{cases}
				g(k,q) & \qquad \text{if $q$ is even or if $|R_k|$ is even}, \\[1mm]
				g(\frac{k}{2},q) & \qquad \text{if $p$ is odd and $|R_k|$ is odd,} 
			\end{cases}
		\end{equation} 
		where $R_k = \{ x^k : x\in \ff_q^*\}$.
	\end{rem}

	More recently, the Waring numbers over finite fields in relation to GP-graphs and their diameters were studied by us in \cite{PV6} and \cite{PV7} (see also \cite{PV9} for the Waring problem over finite commutative local rings).
	
	The next result gives a simple condition for the existence of the weak Waring number $w(k,q)$, in terms of the structure of $\G(k,q)$. 
	In this case, we obtain the same result as Cochrane and Cipra in \eqref{eq: Cipra wkq=gkq}, but using a different method. 
	From the cyclic structure of $\ff_{q}^*$, it can be deduced that 
	$$ w(k,q)=w(k',q) \quad \text{and} \quad g(k,q)=g(k',q) \qquad \text{with} \qquad k'=\gcd(k,q-1), $$ 
	and so we will assume that $k\mid q-1$ when we deal with Waring numbers.
	However, there is a distinction between the directed and undirected cases.
	
	\begin{thm} \label{thm: wkq}
		Let $q=p^m$, with $p$ 
		prime and $m\in \N$, and let $k \in \N$ such that $k\mid q-1$. 
		The number $w(k,q)$ exists if and only if the number $g(k,q)$ exists, which in turn happens if and only if $\G(k,q)$ is connected. 
		In this case we have that	
		\begin{equation} \label{eq: wkq}
			w(k,q) = \begin{cases}
				g(k,q) & \qquad \text{if $q$ is even or if $q$ is odd and $v_{2}(k)<v_{2}(q-1)$}, \\[1mm]
				g(\frac{k}{2},q) & \qquad \text{if $q$ is odd and $v_{2}(k)=v_{2}(q-1)$.} \end{cases}
		\end{equation}
		where $v_2$ denotes the $2$-adic valuation. 
		In other words, 
		$w(k,q) = g(k,q)$ if $\G(k,q)$ is undirected or $w(k,q) = g(\frac k2,q)$ if $\G(k,q)$ is directed. 
	\end{thm}

	\begin{proof}
		If $q$ is even or else if $q$ is odd and $v_{2}(k)<v_{2}(q-1)$, we have that the graph $\G(k,q)$ is undirected. 
		This implies that $-1 \in R_{k}$ and hence 
		$$w(k,q)=g(k,q)$$ 
		in this case. Moreover, Theorem 3.3 from \cite{PV6} implies that $w(k,q)$ exists if and only if $\G(k,q)$ is connected.
		
		Now, assume that $q$ is odd and $v_{2}(k)=v_{2}(q-1)$. By \eqref{eq: G dirs} we have that $\G(k,q)$ is directed and 
		$$\G(\tfrac k2,q) = \overset{_{\rightarrow}}{\G}(k,q) \cup \overset{_{\leftarrow}}{\G}(k,q).$$
		Moreover, in this case $R_{k/2}= R_{k}\cup (- R_{k})$. This implies that if $a\in \mathbb{F}_{q}$ satisfies
		$$a=a_1 x_{1}^{k} + \cdots + a_s x_{s}^{k}$$ 
		with $a_{i} \in \{\pm 1\}$ and $x_i \in \ff_q$ for $1 \le i \le s$, then $a_i x_{i}^{k}= y_{i}^{k/2}$ for some $y_i\in \ff_{q}$ for all $i=1,\ldots,s$ because in this case $-1\in R_{k/2}$ and hence
		\begin{equation}\label{eq: eWn}
			a = {y_1}^{\frac k2} + \cdots + {y_s}^{\frac k2}.
		\end{equation} 
		In particular, $w(k,q)$ exists if and only if $g(\frac{k}{2},q)$ exists. By Theorem 3.3 from \cite{PV6}, the Waring number $g(\frac{k}{2},q)$ exists if and only if $\G(\frac{k}{2},q)$ is connected. By \eqref{eq: Spec Gk2q = 2Re Spec Gkq}, this happens if and only if $\G(k,q)$ is connected, since the multiplicity of the trivial eigenvalue is the same for both graphs. Finally, the equation \eqref{eq: eWn} implies that $w(k,q)=g(\frac{k}{2},q)$, as asserted.
	\end{proof}

	\begin{rem} 
		Since $|R_k|=\frac{q-1}k$, it is clear that the conditions in \eqref{eq: Cipra wkq=gkq} and in \eqref{eq: wkq} are the same.
	\end{rem}

	In \cite{PV6}, we have shown that the Waring number $g(k,q)$ is the diameter of $\G(k,q)$. 
	For this reason we could refer to these GP-graphs as \textit{Waring graphs}.
	Similarly, as a direct consequence  of our previous results, 
	we next show that the weak Waring number $w(k,q)$ is the diameter of the Cayley graph 
	$$W(k,q)=Cay(\ff_q, \check{R}_k), \qquad 
	\check{R}_k=\{\pm x^k : x \in \ff_q \} = R_k \cup (-R_k).$$
	Following the analogy, we can call these graphs \textit{weak Waring graphs}.
	
	\begin{coro}
		If $k\mid q-1$, in the previous notations, we have 
		\begin{equation} \label{eq: wkq diam}
			w(k,q) = {\rm diam}(W(k,q)).
		\end{equation}
	\end{coro}
	
	\begin{proof}
		If $q$ is even or if $q$ is odd and $v_{2}(k)<v_{2}(q-1)$ then $\check{R}_{k}=R_{k}$ and so $W(k,q)=\Gamma(k,q)$. Hence, by the above theorem and Theorem 3.3 from \cite{PV6} we have 
		$$ w(k,q) = g(k,q) = {\rm diam}(\G(k,q)) = {\rm diam}(W(k,q)).$$ 
		
		On the other hand, if $q$ is odd and $v_{2}(k)=v_{2}(q-1)$ then $W(k,q)=\Gamma(\frac{k}{2},q)$, by Theorem~\ref{thm: nature spec v2}. By the previous theorem and Theorem 3.3 from \cite{PV6} we obtain that
		$$
		w(k,q)=g(\tfrac{k}{2},q) = {\rm diam}(\G(\tfrac{k}{2},q)) = {\rm diam} (W(k,q)).
		$$
		Therefore, in any case we get expression \eqref{eq: wkq diam} as we wanted to see.
	\end{proof}

	\begin{rem}
		Notice that if we take $p$ an odd prime and $k=p-1$, then 
		$$ w(p-1,p)=g(\tfrac{p-1}{2},p)=\tfrac{p-1}{2}<p-1 =g(p-1,p), $$
		and hence the weak Waring number and the Waring number not necessarily coincide.
		For instance, in $\Z_5$, we have $1^2=4^2=1$, $2^2=3^2=4$, and hence $1^4=2^4=3^4=4^4=1$.
		We have $w(4,5)=g(2,5)=2$ while $g(4,5)=4$. In fact, 
		\begin{align*}
			& 1=1^4+0^4=1^2+0^2, 			&& 1= 1^4, \\ 
			& 2=1^4+1^4=1^2+1^2,    		&& 2= 1^4+1^4, \\ 	
			& 3=-1^4-1^4=2^2+2^2, 			&& 3= 1^4+1^4+1^4, \\ 	
			& 4=-1^4+0^4=2^2+0^2, 			&& 4= 1^4+1^4+1^4+1^4. 
		\end{align*}
	\end{rem}

	As a direct consequence of the previous theorem and Corollary \ref{coro: wkqs exp} we get the following rephrasing of the theorem written in more precise terms.  In particular, it shows that the weak Waring numbers over binary finite fields always coincide with the Waring number (a result that is clear working in characteristic 2).
	\begin{coro} \label{coro: wkg=gkq}
		In the previous notations, we have. 
		
		\noindent $(a)$ 
		For any fixed $m\in\N$ and any $k\mid 2^m-1$ it holds 
		$$w(k,2^m)=g(k,2^m)$$ 
		if these numbers exist. 
		
		\noindent $(b)$ 
		If $q=2^tr+1$, with $r$ odd, for any $0<t'\le t$ and $s\mid r$ we have 
		\begin{equation} \label{eq: wkqs exp}
			w(2^ts,q) = \begin{cases}
				g(2^ts,q) & \qquad \text{if $t'=t$}, \\[1mm]
				g(2^{t-1}s,q) & \qquad \text{if $t'<t$},
			\end{cases} 
		\end{equation}
		when all these numbers exist.  
	\end{coro}

	\begin{rem}
		As for the Waring number, we can define the weak Waring function from weak Waring pairs. 
		We say that a pair of positive integers $(k,q)$, such that $q$ is a prime power and $k\mid q-1$, is a \textit{weak Waring pair} if $w(k,q)$ exists. We denote by $\mathbb{W}_w$ the set of all such pairs. Consider the \textit{weak Waring function} sending every weak Waring pair to the corresponding weak Waring number, i.e.\@
		\begin{equation} \label{g-function}
			w: \mathbb{W}_w \subset \N \times \N \rightarrow \N, \qquad (k,q) \mapsto w(k,q).
		\end{equation}
		
		Since any positive integer is the Waring number of some pair $(k,q)$ with $-1\in R_k$ (see Proposition 4.7 from \cite{PV7}), we obtain that every positive integer number is the weak Waring number $w(k,q)$ for the pair $(k,q)$, in some (generically non-prime) finite field. 
	\end{rem}

	We now give a reduction formula for the weak Waring numbers, similar to the existing one for Waring numbers (see Theorem 2.4 in \cite{PV7}). We recall that an integer $k$ is a primitive divisor of an integer of the form $p^m-1$ with $p$ prime if $k\mid p^m-1$ and $k \nmid p^t-1$ for any $1\le t < m$. We denote this by 
	$$k \dagger p^m-1.$$
	
	\begin{thm} \label{thm: reduction wkq}
		Let $p$ be a prime and $a, b, c$ positive integers such that $c \dagger p^a-1$ and $bc \dagger p^{ab}-1$. Then, we have 
		\begin{equation} \label{eq: gk=bgu gen}
			w(\tfrac{p^{ab}-1}{bc},p^{ab}) = b w(\tfrac{p^a-1}{c}, p^{a}).
		\end{equation} 
	\end{thm}

	\begin{proof}
		It is enough to assume that $\G(\tfrac{p^{ab}-1}{bc},p^{ab})$ is directed. Indeed, the undirected case 
		(i.e., $q$ even or $q$ odd and $v_2(\tfrac{p^{ab}-1}{bc})<v_2(q-1)$) follows directly from Theorem 2.4 from \cite{PV7} and Theorem \ref{thm: wkq}. 
		
		In the directed case, that is when $q$ is odd and $v_2(\tfrac{p^{ab}-1}{bc})=v_2(q-1)$, Theorem \ref{thm: wkq} implies that 
		\begin{equation} \label{eq: aux fla 1}
			w(\tfrac{p^{ab}-1}{bc},p^{ab}) = g(\tfrac{p^{ab}-1}{2bc},p^{ab}).
		\end{equation}		
		The hypotheses on primitive divisibility clearly imply that $2c \dagger p^a-1$ and $2bc \dagger p^{ab}-1$ also. Therefore, we can apply the reduction formula for Waring numbers (see Theorem~2.4 in \cite{PV7}) and thus we have 
		\begin{equation} \label{eq: aux fla 2}
			g(\tfrac{p^{ab}-1}{2bc},p^{ab})=b g(\tfrac{p^a-1}{2c}, p^{a}).
		\end{equation}
		
		Now, notice that the graph $\G(\tfrac{p^a-1}{c}, p^{a})$ is directed. 
		Indeed, one can invoke the fact that  $\G(\tfrac{p^{ab}-1}{bc},p^{ab})$ is the Cartesian product of $b$ copies of $\G(\tfrac{p^a-1}{c}, p^{a})$ \cite[Proposition 2.3]{PV7}, and the Cartesian product of two graphs is directed if and only if each factor is directed. However, we can give a direct proof of this in terms of the $2$-adic valuation. We need to show that 		
		$$ v_2(\tfrac{p^{ab}-1}{bc})=v_2(p^{ab}-1) 
		\quad \Rightarrow \quad v_2(\tfrac{p^{a}-1}{c}) = v_2(p^{a}-1).$$ 
		Suppose that $v_2(\frac{p^{a}-1}{c}) < v_2(p^{a}-1)$, hence $\frac{p^a-1}{c} \mid \frac{p^a-1}2$, that is $\frac{p^a-1}{c}=\frac{p^a-1}{2} t$ for some $t$. Therefore $c=2t$ and thus $bc$ is even. This implies that $v_2(\frac{p^{ab}-1}{bc}) < v_2(p^{ab}-1)$, which is a contradiction.  
		Since $\G(\tfrac{p^a-1}{c}, p^{a})$	is directed, by \eqref{eq: wkq} again we have that
		1\begin{equation} \label{eq: aux fla 3}
			w(\tfrac{p^a-1}{c}, p^{a})=g(\tfrac{p^a-1}{2c}, p^{a})
		\end{equation}
		Putting together the expressions \eqref{eq: aux fla 1}, \eqref{eq: aux fla 2} and \eqref{eq: aux fla 3} we obtain the reduction formula \eqref{eq: gk=bgu gen} for weak Waring numbers. 
	\end{proof}

	\begin{rem}
		We point out that the reduction formula for weak Waring numbers obtained in Theorem \ref{thm: reduction wkq} together with the relation between weak Waring numbers $w(k,q)$ and Waring numbers $g(k,q)$ and $g(\frac k2,q)$, allow to obtain several explicit expressions and formulas for the numbers $w(k,q)$, in the same vein as the ones obtained in \cite{PV6} and \cite{PV7}. In fact, almost all the results in \cite{PV7} can be adapted for $w(k,q)$.
		However, we will not do this here.
	\end{rem}

	\section{Some worked examples over fixed fields} \label{sec: examples}
	In this section, we consider some GP-graphs over some small fixed finite fields to illustrate the results of the previous sections. 
	We will study the nature of their spectra and we will compute the associated (weak) Waring numbers.
	
	Namely, for $q=5^2, 7^2, 3^4$ and $2^8$, we consider all the GP-graphs $\G(k,q)$. We will identify them as known graphs (when possible) and give their decompositions into isomorphic copies of smaller GP-graphs in the disconnected cases.

	As before, $K_q$, $\mathcal{P}_q$, $L_{q,q}$, $C_q$ respectively denote the complete, Paley, lattice and cycle graphs (and $\vec{\mathcal{P}}_q$, $\vec{C}_q$ the oriented Paley and cycle graphs). The symbol $\sqcup$ denotes disjoint union of graphs and the graphs are shown as the disjoint union of their connected components. Since $q$ will be fixed, to determine the nature of the spectrum (integral, real, or complex) we just check conditions \eqref{eq: nature} and \eqref{eq: nature2}. 
	
	Furthermore, the diameter of the connected graphs $\G(k,q)$ gives the Waring number $g(k,q)$, from which we also obtain the weak Waring number $w(k,q)$. 
	We recall that 
	\begin{equation} \label{eq: diams}
		{\rm diam}(K_q)=1, \qquad {\rm diam}(\mathcal{P}_q) = {\rm diam}(L_{q,q})=2, \qquad {\rm diam}(\G_{srg})=2
	\end{equation}
	where $\G_{srg}$ stands for any strongly regular graph, while ${\rm diam}(C_n)=[\frac n2]$. 
	
	\begin{exam} \label{exam: q=25}
		Since $5^2-1$ has eight divisors, there are eight GP-graphs over $\ff_{5^2}$. From Example 3.6 in \cite{PV18}, they are 
		\begin{align*}
			\G(1,25)  & = K_{25}, \\ 
			\G(2,25)  & = \mathcal{P}_{25}, \\ 
			\G(3,25)  & = srg(25,8,3,2) = L_{5,5}, \\ 
			\G(4,25)  & = \text{undirected, connected, $6$-regular (not srg)}, \\ 
			\G(6,25)  & = K_5 \sqcup K_5 \sqcup K_5 \sqcup K_5 \sqcup K_5, \\ 
			\G(8,25)  & = \text{directed, connected, $3$-regular (not srg)}, \\
			\G(12,25) & = C_5 \sqcup C_5 \sqcup C_5 \sqcup C_5 \sqcup C_5 = 
			\mathcal{P}_5 \sqcup \mathcal{P}_5 \sqcup \mathcal{P}_5 \sqcup \mathcal{P}_5 \sqcup \mathcal{P}_5, \\ 
			\G(24,25) & = \vec{C}_5 \sqcup \vec{C}_5 \sqcup \vec{C}_5 \sqcup \vec{C}_5 \sqcup \vec{C}_5. 
		\end{align*}	
		
		\noindent \textit{Nature of the spectrum.} 
		We see at glance that the graphs $\G(8,25) = \G(2^3,25)$ and $\G(24,25) = \G(2^3\cdot 3,25)$ have complex spectrum since they are directed (while the rest have real spectra). 
		Alternatively, note that 	
		$$ q-1=24=2^3 \cdot 3$$
		and use Theorem \ref{thm: nature spec v2} or Corollary \ref{coro: wkqs exp}.	
		
		Moreover, we see that $\G(1,25)$, $\G(2,25)$, $\G(3,25)$ and $\G(6,25)$ have integral spectrum since $k\mid \frac{25-1}{5-1}=6$, while the remaining graphs $\G(4,25)$ and $\G(12,25)$ have real non-integral spectrum ($k\mid 12$ but $k\nmid 6$.) 
		
		This is in accordance with Corollaries \ref{coro: wkqs exp} and \ref{coro: integral GPs} giving $N_\C(5^2)=2$, $N_\R(5^2)=2 \cdot 3=6$ and $N_\Z(5^2) = 2 \cdot 2 =4$. 
		
		\sk 
		
		\noindent \textit{Waring numbers.}
		From the disconnected graphs we see that the Waring numbers $g(6,25)$, $g(12,25)$ and $g(24,25)$ and the weak Waring numbers $w(6,25)$, $w(12,25)$ and $w(24,25)$ do not exist. 
		From the connected graphs $\G(k,25)$ we know that the corresponding Waring numbers exist and also 
		$g(1,25)=1$ and $g(2,25)=g(3,25)=2$, 
		where we have used \eqref{eq: diams}
		(These values can also be obtained for instance from the expressions in Corollary 3.5 in \cite{PV7}). 
		From List 4 ($c$) in Section 7 of \cite{PV7} 
		\footnote{We point out here that there are some errors in item ($c$) of List 4 in Section 7 of \cite{PV7}. Namely $g(6,25)$ and $g(10,81)$ do not exist since $\G(6,25)$ and $\G(10,81)$ are not connected; while $g(12,81)=g(4,81)=2$ since $\gcd(12,80)=\gcd(4,80)=4$ and $\G(4,81)$ is the Brouwer-Haemers graph which, being strongly regular, has diameter 2.} 
		we see that 
		$g(8,25)=4$.
		
		On the other hand, the number $g(4,25)$ is slippery, since no known result (to the authors' knowledge) give us the explicit value. Either upper bounds for $g(k,p^m)$ in ($a$) or ($b$) of Subsection 2.2 in \cite{PV6}, which are results of Winterhoff from 1998, give $g(4,25) \le 4$. Hence, we compute this number by hand. Notice that 
		$$\mathbb{F}_{25}\cong \mathbb{F}_{5}[x]/(x^{2}+2x+3)\cong\mathbb{F}(\alpha)=\{c\alpha+ d: c,d\in \mathbb{F}_{5}, \alpha^{2}=3\alpha+2\}.$$ 
		One can prove that $\alpha$ is a primitive element of $\mathbb{F}_{25}$ and the set of non-zero $4$th powers is
		$$\{1, 4, \alpha +3, \alpha +4, 4\alpha+1, 4\alpha+2 \}.$$ 
		More precisely,
		$\alpha^{4}=4\alpha+2$, $\alpha^{8}=4\alpha+1$, $\alpha^{12}=4$, $\alpha^{16}=\alpha +3$ and $\alpha^{20}= \alpha+2$. 
		It is straightforward to show that all of the elements of $\mathbb{F}_{25}$ can be written as a sum of tree $4$th powers. Moreover, for instance, the element $\alpha+1$ cannot be written as a sum of two $4$th powers, which implies that 
		$g(4,25)=3$.

		Finally, from Theorem \ref{thm: wkq} we have that $g(k,25)=w(k,25)$ for $1\le k \le 3$ and $g(4,25)=w(4,25)=w(8,25)$.
		Summing up, we have
		\begin{align*}
			&g(1,25)= w(1,25)=1, \\[1mm]
			&g(2,25)= w(2,25)=g(3,25)= w(3,25)=2, \\[1mm]
			&g(4,25)= w(4,25)= w(8,25)=3, \\[1mm]
			&g(8,25)= 4. 
		\end{align*}
		
		So, $w(8, 25) < g(8,25)$ in this case.
		To illustrate this, in the same notation as before, notice that the set of $8$th powers are given by 
		$\{1, \alpha +3, 4\alpha+1\}$.
		If we take, $\beta= 3\alpha+1$, then $\beta$ can be written as a signed sum of three 8th powers in the form  
		$$ \beta = \alpha^{8}+\alpha^{8}-(\alpha^3)^8,$$ 
		but cannot be written with less 8th powers. On the other hand, $\beta$ can be written as a sum of four 8th powers as follows 
		$$ \beta = \alpha^{8}+\alpha^{8}+\alpha^{8}+(\alpha^2)^8, $$ 
		but cannot be written with less 8th powers.
	\end{exam}

	\begin{exam} \label{exam: q=49}
		By Corollaries \ref{coro: wkqs exp} and \ref{coro: integral GPs}, there are ten GP-graphs over $\ff_{7^2}$, out of which two have complex spectrum and four have integral spectrum. 
		From Example 3.7 in \cite{PV18}, these GP-graphs over $\ff_{7^2}$ are given by 
		\begin{align*}
			\G(1,49)  & = K_{49}, \\ 
			\G(2,49)  & = \mathcal{P}_{49}, \\ 
			\G(3,49)  & = \text{undirected, connected, $16$-regular (not srg)}, \\ 
			\G(4,49)  & = srg(49,12,5,2) = L_{7,7}, \\ 
			\G(6,49)  & = \text{undirected, connected, $8$-regular}, \\
			\G(8,49)  & = K_7 \sqcup K_7 \sqcup K_7 \sqcup K_7 \sqcup K_7 \sqcup K_7 \sqcup K_7, \\
			\G(12,49) & = \text{undirected, connected, $4$-regular},  \\ 
			\G(16,49) & = \vec{\mathcal{P}}_7 \sqcup \vec{\mathcal{P}}_7 \sqcup \vec{\mathcal{P}}_7 \sqcup \vec{\mathcal{P}}_7 \sqcup \vec{\mathcal{P}}_7 \sqcup \vec{\mathcal{P}}_7 \sqcup \vec{\mathcal{P}}_7, \\ 
			\G(24,49) & = C_7 \sqcup C_7 \sqcup C_7 \sqcup C_7 \sqcup C_7 \sqcup C_7 \sqcup C_7, \\
			\G(48,49) & = \vec{C}_7 \sqcup \vec{C}_7 \sqcup \vec{C}_7 \sqcup \vec{C}_7 \sqcup \vec{C}_7 \sqcup \vec{C}_7 \sqcup \vec{C}_7.
		\end{align*}
		
		\noindent \textit{Nature of the spectrum.}  
		The graphs $\G(16,49) = \G(2^4,25)$ and $\G(48,49) = \G(2^4\cdot 3,49)$ have complex spectrum since they are directed. Also, note that 
		$$ q-1=48=2^4 \cdot 3$$ 
		and use Theorem \ref{thm: nature spec v2} or Corollary \ref{coro: wkqs exp}.
		Moreover, we see that $\G(1,25)$, $\G(2,25)$, $\G(4,25)$ and $\G(8,25)$ have integral spectrum since 
		$k\mid \frac{49-1}{7-1}=8$ and the remaining graphs have real non-integral spectrum ($k\mid 12$ but $k\nmid 6$.) 
		
		\sk
		
		\noindent \textit{Waring numbers.}	
		First, from the disconnected graphs in the list, we see that the Waring numbers $g(k,49)$ and weak Waring numbers $w(k,49)$ for $k=8,16,24,48,$ do not exist. 
		
		On the other hand, by Theorem \ref{thm: wkq}, we have that $w(k,49)=g(k,49)$ for $k=1,2,3,4,6$. By using the diameter of the graphs, it is immediate that
		$$ w(1,49)=g(1,49)=1 \quad \text{and} \quad  
		w(2,49)=g(2,49)=w(4,49)=g(4,49)=2.$$
		
		Now, Corollary 3.12 in \cite{PV7} ensures that 
		$$ g(\tfrac{p^2-1}4, p^2)=p-1$$
		for every odd prime $p\equiv 3 \pmod 4$,m and hence we have $$ w(12,49)=g(12,49)=6.$$ 
		
		For the remaining values of $k$, i.e.\@ $k=3,6$, the results in \cite{PV6} and \cite{PV7} seem not to apply, and hence we compute them by hand. 
		
		Since $\mathbb{F}_{49}$ is isomorphic to 
		$$\{a+b\alpha: a,b \in \mathbb{Z}_{7}, \, \alpha^{2}=-1\}.$$ 
		By direct computation, we get the set of cubes in $\ff_{49}$ which is
		\begin{multline*}
			\{0,1, \,-1,\,\alpha, \,-\alpha, \,2+2\alpha, \, 2-2\alpha, \, 2+4\alpha, \, 2-4\alpha,\, -2+2\alpha,\, -2-2\alpha, \\
			-2+4\alpha, \,-2-4\alpha, \, 4+2\alpha, \, 4-2\alpha, \, -4+2\alpha, \, -4-2\alpha\}.
		\end{multline*}
		A straightforward computation with the above set allow one to obtain that 
		$$w(3,49)=g(3,49)=2,$$ 
		that is any element of $\ff_{49}$ is a sum of two cubes. 
		Similarly, the set of $6$-th powers is
		$$
		\{0,1, \,-1,\,\alpha, \,-\alpha, \,2+2\alpha, \, 2-2\alpha, -2+2\alpha,\, -2-2\alpha\},
		$$
		in this case, by direct computation we can obtain that any element of $\ff_{49}$ is a sum of three $6$-th powers and hence $w(6,49)=g(6,49)=3$.
	\end{exam}

	\begin{exam} \label{exam: q=81}
		There are ten GP-graphs over $\ff_{3^4}$. From Example 3.5 in \cite{PV18}, they are given by 
		\begin{align*}
			\G(1,81)  & = K_{81}, \\ 
			\G(2,81)  & = \mathcal{P}_{81}, \\ 
			\G(4,81)  & = srg(81,20,1,6) = \text{Brouwer-Haemers graph}, \\ 
			\G(5,81)  & = srg(81,16,7,2) = L_{9,9}, \\ 
			\G(8,81)  & = \text{undirected, connected, $10$-regular (not srg)}, \\ 
			\G(10,81) & = 9K_9, \\
			\G(16,81) & = \text{directed, connected, $5$-regular (not srg)}, \\ 
			\G(20,81) & = 9 \mathcal{P}_9, \\
			\G(40,81) & = 27 C_3 = 27 K_3, \\ 
			\G(80,81) & = 27 \vec{C}_3 = 27 \vec{\mathcal{P}}_3. 
		\end{align*}

		\noindent \textit{Nature of the spectrum.}  
		The graphs $\G(16,81)$ and $\G(80,81)$ have complex spectrum since they are directed (or note that $q-1=80=2^5 \cdot 5$ and use Theorem \ref{thm: nature spec v2} or Corollary~\ref{coro: wkqs exp}).
		The remaining GP-graphs are integral, since $\frac{q-1}{p-1}=\frac{3^4-1}{2}=40= 2^3 \cdot 5$
		and, hence, by Corollary \ref{coro: integral GPs} we have that $N_\Z(3^4) = 4\cdot 2=8$. 
		
		\sk
		
		\noindent \textit{Waring numbers.}
		First, from the disconnected graphs in the list, we see that the Waring numbers $g(k,81)$ and weak Waring numbers $w(k,81)$ for $k=10,20,40,80$ do not exist. 
		
		On the other hand, by Theorem \ref{thm: wkq}, we have that $w(k,81)=g(k,81)$ for $k=1,2,4,5,8$ and $w(16,81)=g(8,81)$
		By using the diameter of the graphs, it is immediate that
		\begin{align*}
			& w(1,81)=g(1,81)=1, \\  
			& w(2,81)=g(2,81)=w(4,81)=g(4,81)=w(5,81)=g(5,81)=2.	
		\end{align*}
		
		Thus, it only remains to compute $g(8,81)$ and $g(16,81)$. 
		To this end, we can use the expressions due to Winterhof and van de Woestijne (\cite{Win2}), asserting that
		$$ g(\tfrac{p^{r-1}-1}{r},p^{r-1}) = \tfrac 12 (p-1)(r-1) $$
		where $p,r$ are primes, $p$ is a primitive root modulo $r$ and $\varphi$ denotes the Euler totient function. 
		In addition, if $p,r$ are odd, we have 
		$$ g(\tfrac{p^{r-1}-1}{2r},p^{r-1}) = \lfloor \tfrac{pr}4 - \tfrac{r}{4p}  \rfloor $$
		where $r\ge p$.

		Taking $p=3$ and $r=5$, the first expression gives 
		$$ g(16,81) = g(\tfrac{3^4-1}{5},3^4) = \tfrac{1}{2}(3-1)(5-1) = 4, $$
		while the second expression gives 
		$ g(8,81) = \lfloor \tfrac{15}4 - \tfrac{5}{12} \rfloor = \lfloor \tfrac{10}3 \rfloor = 3$, 
		and hence
		$$ g(8,81) = w(8,81) = w(16,81) = 3, $$
		completing the computation of the (weak) Waring numbers in this case. Here, we have $w(16,81)<g(16,81)$.
	\end{exam}

	\begin{exam} \label{exam: q=256}
		There are eight GP-graphs over $\ff_{2^8}$ given by the divisors of $2^8-1$.
		From Example 4.5 in \cite{PV18}, they are given by
		\begin{align*}
			\G(1,256)  & = K_{256}, \\ 
			\G(3,256)  & = \text{connected $85$-regular} = srg(256, 85, 24, 30), \\ 
			\G(5,256)  & = \text{connected $51$-regular} = srg(256, 51, 2, 12), \\ 
			\G(15,256) & = \text{connected $17$-regular (not srg)}, \\ 
			\G(17,256) & = K_{16} \sqcup \cdots \sqcup K_{16} \quad (\text{$2^4$-times}), \\ 
			\G(51,256) & = \G(3,16) \sqcup \cdots \sqcup \G(3,16) \quad (\text{ $2^4$-times}). \\
			\G(85,256) & = K_4 \sqcup \cdots \sqcup K_4 \quad (\text{$2^6$-times}), \\ 
			\G(255,256)& = K_2 \sqcup \cdots \sqcup K_2 \quad (\text{$2^7$-times}). 
		\end{align*}

		\noindent \textit{Nature of the spectrum.}  
		By \eqref{eq: 2^m integral} in Remark \ref{rem: nature}, all these binary GP-graphs are integral. 
		
		\sk 
		
		\noindent \textit{Waring numbers.}
		The (weak) Waring numbers do not exist for $k=17, 51, 85, 255,$ since the associated graphs $\G(k,q)$ and $W(k,q)$ are disconnected in these cases. 
		For the remaining values of $k$, by Theorem \ref{thm: wkq} or $(a)$ of Corollary \ref{coro: wkg=gkq} we have that 
		$$ w(k,2^8)=g(k,2^8) \qquad \text{for $k=1, 3, 5, 15$.}$$ 
		Also, by using the diameter of the graphs, we have that 
		$$ g(1,256)=1 \qquad \text{and} \qquad g(3,256)=g(5,256)=2.$$
		It remains to compute $g(15,256)$.
		First note that, by a result of Glibichuk and Rudnev (\cite{GlR}), $g(k,q)\le 8$ if $k< \sqrt{q}$ , and hence we have that $g(15,256)\le 8$.
		Now, since the polynomial $p(x)=x^{8}+x^{4}+x^{3}+x+1$ is irreducible over $\mathbb{F}_{2}$, we have that 
		$$ \mathbb{F}_{256} \simeq \mathbb{F}_{2}[x]/(p(x)) = \{a_{0}+a_1 \alpha + \cdots +a_{7}\alpha^{7}: a_{i}\in \ff_2, \, \alpha^{8}=\alpha^{4}+\alpha^{3}+\alpha+1\}. $$ 
		where $\alpha$ is a root of $p(x)$. 
		By direct computation, if $\omega=\alpha^{15}$ then $\omega=\alpha^5+\alpha^3+\alpha^2+\alpha+1$ and $\omega^{7}=\alpha^{3}$. These equalities and the property $(a+b)^{2}=a^{2}+b^{2}$ allow us to obtain the set of $15$-th non-zero powers, which are 
		$$
		\omega^{0}=1, \, \omega^{1}= \alpha^5+\alpha^3+\alpha^2+\alpha+1, \, \omega^{2}= \alpha^5+\alpha^4+\alpha^3+1, \, \omega^{3}= \alpha^4+\alpha^3+\alpha^2+1
		$$
		$$
		\omega^{4}= \alpha^5+\alpha^4+\alpha^2+\alpha, \,\omega^{5}= \alpha^7+\alpha^5+\alpha^4+\alpha+1, \, \omega^{6}= \alpha^6+\alpha^3+\alpha,\, \omega^{7}= \alpha^3, 
		$$
		$$
		\omega^{8}= \alpha^6+\alpha^5+\alpha+1, \, \omega^{9}= \alpha^7+\alpha^6+\alpha^4+\alpha+1, \, \omega^{10}= \alpha^7+\alpha^6+\alpha^5+\alpha^3,
		$$
		$$
		\omega^{11}= \alpha^7+\alpha^5+\alpha^3+\alpha+1, \, \omega^{12}= \alpha^7+\alpha^6+\alpha^5+\alpha^3+\alpha^2+\alpha+1, \, \omega^{13}= \alpha^5+\alpha^4+\alpha^3+\alpha^2, 
		$$
		$$
		\omega^{14}= \alpha^6, \, \omega^{15}= \alpha^5+\alpha^4+\alpha^2+1, \, \omega^{16}= \alpha^7+\alpha^6+\alpha.
		$$
	\end{exam}
	
	It is straightforward to see that $\alpha^{2}$ cannot be written as a sum of two $15$-th powers and so $g(15,256)\ge 3$. 
	Thus, have that 
	$$3 \le g(15,256) \le 8.$$
	However, all known results for exact values cannot be applied or fail to give an answer, and no known bounds seem to improve the previous one.  
	Nevertheless, by using Python, one can check that 
	$$ g(15,256)=3.$$
	This shows that, in general, it is hard to compute exact values of Waring numbers without using some mathematical software.

\end{document}